\documentclass[11pt]{article}
\usepackage{amsfonts}
\usepackage{amsmath,amsthm,mathtools,amscd,amssymb,mathrsfs,setspace}
\usepackage{latexsym, epsf, epsfig}
\usepackage{color}
\usepackage[hmargin=.95in,vmargin=.95in]{geometry}
\usepackage{graphicx}
\usepackage{hyperref}
\usepackage{cite}
\usepackage{multirow}
\usepackage{xcolor}

\usepackage{float} 
\usepackage{cancel} 
\newcommand{\nn}{\nonumber}

\newcommand{\cA}{{\mathcal{A}}}

\newcommand{\cE}{{\mathcal{E}}}

\newcommand{\bu}{\mathbf u}
\newcommand{\bv}{\mathbf v}
\newcommand{\bw}{\mathbf w}
\newcommand{\bF}{\mathbf F}

\newcommand{\bV}{\mathbf V}

\newcommand{\grad}{\nabla}
\newcommand{\Om}{\Omega}

\renewcommand{\vec}[1]{\mathbf{#1}}

\newcommand{\btau}{\boldsymbol{\tau}}
\newcommand{\bxi}{\boldsymbol{\xi}}
\newcommand{\bzeta}{\boldsymbol{\zeta}}
\newcommand\restri[1]{\left.{#1}\right|_{\Gamma_I}}
\newcommand\weakto\rightharpoonup

\theoremstyle{plain}
\newtheorem{theorem}{Theorem}[section]

\newtheorem{corollary}[theorem]{Corollary}

\newtheorem{definition}{Definition}
\theoremstyle{remark}
\newtheorem{remark}{Remark}[section]
\numberwithin{equation}{section} \numberwithin{theorem}{section}
\numberwithin{remark}{section} 
\providecommand{\doititle}[1]{#1}

\begin{document}

\title{Weak Solutions and Inertial Limits for Quasi-static Filtrations}
\author{{\footnotesize
\begin{tabular}[t]{c@{\extracolsep{3em}}c@{\extracolsep{3em}}c}
    Peter Lavagnino & Arum Lee & Justin T. Webster \\
    \it UMBC \hskip1.3cm &
      \hskip.4cm \it UMBC \hskip1.3cm &
      \hskip.4cm \it UMBC\\
    \it Baltimore, MD & \it Baltimore, MD & \it Baltimore, MD \\
    \texttt{lavagn1@umbc.edu} & \texttt{leearum9@umbc.edu} & \texttt{websterj@umbc.edu} \\
\end{tabular}}}
\maketitle

\begin{abstract}
\noindent  A quasi-static filtration system, comprising a poroelastic solid coupled to an incompressible free-flow, is considered in 3D. Across a flat 2D interface, the Beavers-Joseph-Saffman coupling conditions are taken. The system constitutes a doubly elliptic-parabolic coupling and can be seen as a degenerate case of the inertial Biot-Stokes dynamics. These dynamics cannot be easily recovered by simply sending the inertial parameters to zero; however, inserting a viscoelastic regularization of the inertial Biot system allows us to construct weak solutions in the inertial limit. Subsequently, we can pass to the limit in the regularization parameter to obtain quasi-static weak solutions. This addresses an open singular/degenerate limiting problem in the class of filtration models, and permits future analyses of uniqueness and regularity of solutions. This linear construction also provides a foundation for the later incorporation of physically-motivated nonlinear poroelastic effects via fixed point methods.

\vskip.25cm
\noindent {\bf Acknowledgments}: The second author was partially supported by NSF-DMS 2307538. The third author was partially supported by NSF-DMS 2606567.
\end{abstract}

\section{Introduction}
In this treatment we consider a linear filtration system of a fluid free-flow and a saturated poroelastic solid. Our primary motivation is the work in \cite{showfiltration}, where an early analysis of filtrations is provided in the inertial and non-degenerate regime (with fluid compressibility). In that work, quasi-static and degenerate regimes are mentioned in the context of ``singular limits" \cite[p.13]{showfiltration}, but solutions are not constructed for quasi-static, incompressible dynamics. Unlike some previous work (such as \cite{filtration2} or that  involving the third author \cite{AGW,AW}), our primary focus here is therefore on the {\em fully degenerate} filtration system---that is, a filtration which is entirely quasi-static and incompressible. This system of interest is thus doubly elliptic-parabolic and coupled across a lower-dimensional interface via dynamic and mixed conditions of a dissipative nature: the so-called Beavers-Joseph-Saffman slip conditions \cite{mikelicBJS}. This system exhibits several low-regularity features, described below, which make its analysis challenging. We present the equations here, but relegate a detailed discussion of their structure to Section \ref{coupled1}. 

Denoting $\Omega_b$ as the poroelastic domain and $\Omega_f$ as the free fluid domain (bounded subsets of $\mathbb R^3$), the interior dynamics evolve according to the (elliptic-degenerate parabolic-elliptic) equations
\begin{align}\label{introsystem}
   - \mu\Delta\vec{u} - (\lambda + \mu)\nabla(\nabla\cdot\vec{u}) + \alpha \nabla p_b = \vec{F}_b, &\text{ in } \Omega_b \times (0,T),\\ \label{biot2}
    [c_0p_b + \alpha \nabla \cdot \vec{u}]_t - \nabla \cdot [k\nabla p_b] = S, &\text{ in } \Omega_b \times (0,T)\\
    -\nu\Delta \bv +\nabla p_f=\mathbf F_f;~~\nabla \cdot \bv = 0, & \text{ in } \Omega_f \times (0,T).
\end{align}

\begin{figure}[htbp]
    \centering
    \includegraphics[width=.55\textwidth]{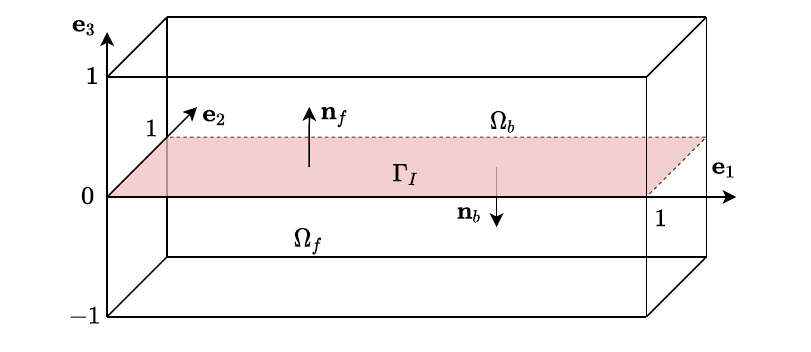}
    \caption{Simplified geometric configuration of the coupled Biot-Stokes filtration model for the poroelastic domain $\Omega_b=(0,1)^3$ and the adjacent fluid domain $\Omega_f=(0,1)\times(0,1)\times(-1,0)$. The two domains meet at the interface $\Gamma_I$, where $\vec{n}_f=\vec{e}_3$ denotes the unit outward normal to $\Omega_f$. Future work will address more general geometries. 
 }
    \label{fig:biot-stokes-geometry}
\end{figure}

The Biot variables are $\bu$ and $p_b$, denoting poroelastic displacement and pressure (resp.), and the fluid variables are $\bv$ and $p_f$, likewise denoting fluid velocity and pressure.
The parameters $\lambda,\mu>0$ are the poroelastic structure's Lam\'e coefficients, while $k >0$ measures the permeability of the porous matrix. The coefficient $c_0\ge 0$ represents the storage coefficient for the Biot dynamics and $\alpha >0$ the Biot-Willis constant \cite{coussy}. The kinematic viscosity of the free flow is $\nu>0$.

Across a 2D interface $\Gamma_I$ separating $\Omega_f$ and $\Omega_b$ ($\Gamma_I \subset \partial \Omega_f \cap \partial\Omega_b$), we have the conditions
\begin{align}
    k\nabla p_b\cdot\mathbf n_f &=- \vec{v} \cdot \mathbf n_f +
    \vec{u}_t \cdot \mathbf n_f, \label{IC1}\\
    \beta(\vec{v} - \vec{u}_t)\cdot \btau &= \btau \cdot \sigma_f \mathbf n_f,\label{IC2}\\
    \sigma_f\vec{e}_3 &= \sigma_b\vec{e}_3,\label{IC3}\\
    p_b &= -\mathbf n_f \cdot \sigma_f\mathbf n_f 
    \label{IC4}
\end{align}
\noindent We use the notation of  $\btau$ generically for  tangential vectors on $\Gamma_I$ and we let  $\mathbf n_f$ represent the unit outward normal to the fluid domain $\Omega_f$. Above, $\sigma_b$ and $\sigma_f$ represent poro(visco)elastic and fluid stress (resp.) tensors.

The central challenges in the analysis revolve around the implicit degeneracy in the problem, as well as the nature of the coupling (low regularity). The most obvious challenge in the analysis of Biot dynamics concerns the regularity of the poroelastic velocity $\bu_t$. Standard approaches (also for thermoelastic dynamics) necessitate using $\bu_t$ as a test function to obtain an a priori (energy) identity; however, weak solutions corresponding to the energy identity do not support any regularity of that very velocity. This is the heart of {\em degeneracy} in this context. 
For this reason, construction of solutions is problematic, and while good a priori estimates exist for the degenerate (limit system), a weak solution does not possess the requisite regularity to obtain the estimates from the weak formulation. This is a problem which has been noted several times in recent literature concerning hyperbolic-parabolic models \cite{rectplate,multilayered, sunny3} and, in particular, poroelastic dynamics \cite{bmw,bw}. There are several approaches which may be viable, involving implicit degenerate theory or temporal/spatial discretization to provide a construction. Here we opt to rely on recently established semigroup theory in the inertial and non-degenerate regimes. We therefore view the central dynamics of interest as singular limits of the inertial system.

To rigorously utilize previous works on the inertial filtration dynamics \cite{AGW,AW,bmw2}, we  adopt a vanishing inertial limit approach for the construction of weak solutions. In general, this is challenging for hyperbolic models, or any models without a strong notion of intrinsic dissipation. To circumvent associated challenges, we will introduce viscoelasticity (Kelvin-Voigt/strong, dissipation) into the Biot dynamics \cite{bmw2}, as is now common in their analysis \cite{bgsw, rectplate,sunny3}. The auxiliary model we utilize is
\begin{align}\label{introsystem2}
 \rho_b\bu_{tt}  - \mu\Delta(\vec{u}+\delta\bu_t) - (\lambda + \mu)\nabla(\nabla\cdot(\vec{u}+\delta \bu_t)) + \alpha \nabla p_b = \vec{F}_b, &\text{ in } \Omega_b \times (0,T),\\ \label{biot2*}
    [c_0p_b + \alpha \nabla \cdot \vec{u}]_t - \nabla \cdot [k\nabla p_b] = S, &\text{ in } \Omega_b \times (0,T)\\
  \rho_f\bv_t  -\nu\Delta \bv +\nabla p_f=\mathbf F_f;~~\nabla \cdot \bv = 0, & \text{ in } \Omega_f \times (0,T).
\end{align}
Here $\rho_b,\rho_f >0$ are the inertial parameters, and $\delta>0$ is a parameter which measures the strength of linear viscoelasticity present in Biot's dynamics. 
The damped, inertial dynamics admit a semigroup representation, and thus strong and mild (or semigroup) solutions are available \cite{pazy}---this is described in depth in the Appendix. The semigroup will permit us to construct quasi-static solutions via the aforementioned two-fold inertial limit. The use of Kelvin-Voigt regularization is motivated especially by the recent analysis of Biot dynamics in \cite{bgsw,bmw2}, where viscoelasticity provides additional time regularity, especially in degenerate and quasi-static regimes. Here the regularization is used not as part of the final model, but as an auxiliary device for obtaining a limit which is rather inaccessible from  baseline ``undamped'' a priori estimates. In the presence of viscoelasticity, one can pass to the limit in $\rho_b$ and $\rho_f$, holding $\delta>0$ fixed. Subsequently, a priori estimates can be obtained from the quasi-static, viscoelastic system (independent of the inertial parameters), and in a final step, the damping parameter $\delta$ can be sent to zero.

This strategy reflects that, at the level of limit passage, the strength of damping is a central issue in degenerate dynamics. Indeed, in an inertial system with hyperbolic-like components one must have control of the wave-type velocity, and said control may not be available from  energy estimates. One may hope to recover relevant estimates indirectly through mechanisms analogous to those appearing in thermoelasticity \cite{thermo}, but such arguments are strongly tied to geometry and often require methods of microlocal analysis. Our route through viscoelastic regularization is more direct. Standard baseline energy estimates for the undamped and quasi-static system are sufficient for the final construction of weak solutions for the problem at hand: $\delta=\rho_b=\rho_f=0$. As a byproduct of our approach, we will produce some auxiliary results for the viscoelastic filtration system, akin to those in \cite{bgsw} for the isolated poroelastic dynamics.

The literature on fluid-poroelastic interaction has developed along several related, but technically distinct, lines. The foundational poromechanics background goes back to the consolidation theory of Biot and Terzaghi \cite{biot,biot2,terzaghi}, with modern continuum treatments in, for instance, \cite{coussy,poroapps}. On the PDE side, Showalter's early work on implicit and degenerate evolution equations and poroelasticity \cite{indiana,show2000} provides a natural abstract framework for treating Biot's equations in several regimes, while later filtration work coupled poroelastic dynamics to  fluid free flows \cite{showfiltration}. More recently, weak-solution theories for nonlinear and degenerate poroelasticity have been developed in \cite{bgsw,bw,bmw,bmw2}, including models with incompressible constituents and with viscoelastic regularization. Even for stand-alone Biot dynamics, time regularity in the quasi-static setting is subtle \cite{bw,bmw,bmw2}; in the present 3D-3D multiphysics setting, the low-regularity interface conditions add a layer of difficulty. Multilayered and moving-boundary  fluid-poroelastic structure interactions, motivated in part by biomechanics, have been studied in \cite{multilayered,rectplate,sunny3}; related reduced or application-driven models appear in \cite{canicbio1,sunny1,sunny2,GGbook}. These works form a broader background for the present analysis: the exact filtration problem studied here is linear and posed on fixed domains, but the main analytical difficulty arises from the degenerate, mixed-type character of these dynamics. The low regularity intrinsic to the system confounds standard approaches to constructing weak solutions from baseline energy estimates.  

There is also a substantial numerical and modeling literature for Stokes-Biot and Navier-Stokes-Biot couplings. This includes Lagrange-multiplier and mixed finite element formulations \cite{yotovLagrange,yotov2022,yotovMultipoint,yotovNSB}, nonlinear  variants \cite{filtration2}, partitioned and operator-splitting methods \cite{bukac,filtnum}, and other formulations \cite{yotovBSeye}. The adjacent Stokes-Darcy and Navier-Stokes-Darcy literature, including \cite{infsup1,mikelicBJS}, supplies some mathematical infrastructure for Beavers-Joseph-Saffman interface conditions;  homogenization and dimension-reduction studies such as  \cite{auriault,Sanchez-Palencia,jagerr,jnr,gilbert3,gilbert4} provide additional context for poroelastic models, filtration models, and applications. 

 The work at hand does not pursue modeling, numerical approximation, or homogenization. Rather, we isolate the quasi-static Biot-Stokes filtration problem with incompressible free flow and develop a weak-solution theory by means of viscoelastic regularization in the inertial limit. We believe this is a first construction of weak solutions, based on finite energy considerations, in the case of quasi-static filtrations. Although linear, the construction we present here is a necessary step toward nonlinear extensions. In particular, following the works \cite{bgsw,multilayered}, one must have a viable linear solution theory in hand---robust to variable permeability---in order to extend the weak solution theory through fixed point or other such methods \cite{bw,bmw}. More broadly, as inertial effects are central when considering nonlinear coupling in multiphysics systems (e.g. \cite{sunny3}), and other mixed low-regularity systems, such as magneto-elastic or fluid-structure problems, it is essential to understand when and how quasi-static weak solutions can be obtained through inertial limits. We provide further discussion of the main result and its relation to the existing literature in Section \ref{subsection:main_result}.

The remainder of the paper is organized as follows. Section~\ref{coupled1} provides a detailed description of the coupled Biot--Stokes filtration model, including the geometry, boundary conditions, and Beavers--Joseph--Saffman interface conditions. Section \ref{section:energy_space_weak_solution} introduces the relevant energy identities and function spaces and gives the definition of weak solutions used throughout the paper. Section \ref{section:main_result} states the main existence theorem---theorem \ref{th:main1}---for the quasi-static filtration problem and reviews the supporting well-posedness results for the inertial, viscoelastic system. In Section~\ref{section:proof_of_main_theorem}, we prove the main result through a sequence of singular limits: first, the inertial parameters $\rho_b$ and $\rho_f$ are sent to zero with the viscoelastic parameter $\delta>0$ fixed; subsequently, $\delta$ is sent to zero to recover the fully quasi-static, undamped system, with the storage-degenerate case $c_0=0$ obtained by a final limiting argument. The Appendix provides the semigroup formulation for the inertial, viscoelastic Biot--Stokes system, including the pressure elimination, the definition of the associated generator, and its operator domain.

\section{Detailed Description of the Model}\label{coupled1}
We consider Biot-Stokes dynamics and now provide some background. We follow the exposition  in \cite{GGbook,show2000,bgsw}, and note that the physical configuration at hand is described in \cite{AGW,AW} (see also \cite{showfiltration,filtration2}). Recall that $\Omega_b \subset \mathbb{R}^3$ denotes the fully-saturated poroelastic structure, which we assume to be an isotropic and homogeneous porous medium, undergoing small displacements. In this scenario, such dynamics are  modeled by Biot's equations \cite{biot,biot2,terzaghi,coussy}. The function $\bF_b$ represents a volumetric force on the elastic matrix, and $S$ represents a {fluid source}. We assume the elastic  stress  $\sigma^E(\vec{u})$ obeys the linear strain-displacement law \cite{kesavan,show2000} given by
\begin{equation}\label{elasticstress}
    \sigma^E(\vec{u}) = 2\mu\vec{D}(\vec{u}) + \lambda(\nabla\cdot\vec{u})\vec{I},
\end{equation}
where $\vec{D}(\vec{u}) = \frac{1}{2}(\nabla\vec{u} + (\nabla\vec{u})^T)$ is the symmetrized gradient \cite{kesavan}.

Then the Biot component of the system on $\Omega_b$ is modeled by the following equations:
\begin{align}\label{biot1**}
\rho_b\bu_{tt}   - \text{div}~\sigma^E(\bu)-\delta \text{div}~\sigma^E(\bu_t) + \alpha \nabla p_b = \vec{F}_b, &\text{ in } \Omega_b \times (0,T),\\ \label{biot2**}
    [c_0p_b + \alpha \nabla \cdot \vec{u}]_t - \nabla \cdot [k\nabla p_b] = S, &\text{ in } \Omega_b \times (0,T).
\end{align}
We recall the physical parameters: $\lambda,\mu>0$ are the  Lam\'e coefficients of elasticity \cite{kesavan}, while $\rho_b \ge 0$ captures the Biot mass density; $\delta \ge 0$ measures Kelvin-Voigt damping in the homogenized matrix \cite{bgsw,bmw2};
     $\alpha>0$ is the  {Biot-Willis}  constant \cite{biot2,coussy}, scaled to the system at hand;
     $c_0\ge 0$ is the {storage coefficient} corresponding to  net {compressibility of constituents} ($c_0=0$ represents {incompressible constituents}) \cite{coussy,show2000};
     $k > 0$ is the permeability of the porous matrix.
The {fluid content} of the system is given  by ~$ \zeta = c_0p_b + \alpha \nabla \cdot\vec{u},$
and measures the local fluid mass \cite{show2000}. The discharge velocity (or Darcy flux) is $\vec{q}$, given through Darcy's law $\vec{q} = -k\nabla p_b$. 
 
The free fluid in $\Omega_f$ is modeled by the Eulerian velocity $\vec{v}$, fluid pressure $p_f$, and fluid source $\mathbf F_f$:
\begin{equation}\label{stokes1}
 \rho_f\bv_t  - 2\nu \text{div}~\mathbf D(\vec{v}) + \nabla p_f = \vec{F}_f,\quad \quad \nabla\cdot\vec{v} = 0, \quad  \text{ in } \Omega_f \times (0,T).
\end{equation}
Again, $\rho_f \ge 0$ represents the mass density of the free fluid, and $\nu>0$ the kinematic viscosity.

{\bf As a standing hypothesis for the remainder of the paper, we take $\nu,\alpha, \lambda,\mu>0$, while $\rho_f,\rho_b, \delta, c_0 > 0$, with the  possibility of each of the latter vanishing.}
As with previous recent work \cite{multilayered, AW,AGW}, we simplify the topology of the filtration domain to focus on the interactive dynamics across the interface, $\Gamma_I$. Future work will address arbitrary, physical geometries and the associated technical issues surrounding regularity, boundary triple points, and boundary traces. Namely, we here identify  $x_1$ and $x_2$  sides through laterally-periodic boundary conditions. 
 
 Now---and for the remainder of the paper---we take $\Omega_b \equiv (0,1)^3$ to be the fully-saturated poroelastic structure. 
Let $\Omega_f \equiv  (0,1)\times (0,1) \times (-1,0)$ be a region adjacent to $\Omega_b$,  filled with the fluid described by the incompressible Stokes equations. The geometric configuration of the two domains is illustrated above in Figure~\ref{fig:biot-stokes-geometry}.
The two regions are adjoined at an interface $\Gamma_I = \partial\Omega_b \cap \partial \Omega_f = (0,1)^2$. We denote the normal vectors going out of the Biot and Stokes regions by $\vec{n}_b$ and $\vec{n}_f$, respectively, with  $\vec{n}_f  = \vec{e}_3 = -\vec{n}_b$ on $\Gamma_I$. Denote the Biot and Stokes boundaries by
\begin{equation} \Gamma_b = \{\phi \in \partial \Omega_b~:~x_3=1\}~\text{ and }~ \Gamma_f = \{\phi \in \partial \Omega_f~:~x_3=-1\},\end{equation} 
with homogeneous  conditions on these boundaries:
\begin{equation}\label{BC1}
    \vec{u} = \vec{0} ~~\text{ and }  ~~p_b = 0 ~\text{ on }\Gamma_b, \quad \text{ and } ~~\vec{v} = \vec{0} ~\text{ on } \Gamma_f, \quad \forall\, t \in (0,T),
\end{equation}
and periodic conditions on the four lateral faces (i.e., in the $x_1$ and $x_2$ directions). 

Denoting the {\em total} poroelastic and fluid stress tensors as
\begin{align*}
    \sigma_b = \sigma_b(\vec{u},p_b) &\equiv  \sigma^E(\vec{u}) +\delta \sigma^E(\bu_t)- \alpha p_b\vec{I}, \quad \quad \sigma_f = \sigma_f(\vec{v},p_f) \equiv 2\nu\vec{D}(\vec{v}) - p_f\vec{I},
\end{align*}
respectively, we adapt the interface conditions to our conventions on $\Gamma_I\times (0,T)$:
\begin{align}
     k\nabla p_b\cdot\vec{e}_3
    &= -\vec{v}\cdot\vec{e}_3
    + \vec{u}_t\cdot\vec{e}_3 \label{IC1*}\\
    \beta(\vec{v} - \vec{u}_t)\cdot \btau &= -\btau \cdot \sigma_f \vec{e}_3,\label{IC2*}\\
    \sigma_f\vec{e}_3 &= \sigma_b\vec{e}_3,\label{IC3*}\\
    p_b &= -\vec{e}_3 \cdot \sigma_f\vec{e}_3 
    .\label{IC4*}
\end{align}
 The slip condition in \eqref{IC2*} says that the tangential stress is proportional to the {slip rate}, with slip-length $\beta>0$ (see \cite {mikelicBJS}, and references therein). The conservation of fluid mass across the interface is given by \eqref{IC1*}, the {kinematic coupling condition}. The balance of  stresses in \eqref{IC3*} is required by the conservation of momentum. The {dynamic coupling condition} \eqref{IC4*} maintains the balance of the stress normals across $\Gamma_I$.

\section{Energies, Spaces, and Weak Solutions}\label{section:energy_space_weak_solution}

\subsection{Functional Notation}

We generally work in the framework of $L^2(U)$, where $U \subseteq \mathbb R^n$ for $n=2,3$ is a spatial domain. We denote  $L^2(U)$ inner products by $(\cdot,\cdot)_U$.   Standard Sobolev spaces of the form $H^s(U)$ and $H^s_0(U)$ (along with their duals) will be defined in the typical way \cite{kesavan}, with the $H^s(U)$ norm denoted by $||\cdot||_{s,U}$, or just $||\cdot||_{s}$. For a Banach space $Y$ we denote its dual as $Y'$, and  the associated  pairing as $\langle \cdot, \cdot\rangle_{Y'\times Y}$. We denote $\vec{x} = (x_1,x_2,x_3) \in \mathbb{R}^3$, with associated spatial differentiation by $\partial_i$. {For estimates, we will use the notation $A\lesssim B$ to mean that there exists a constant $c$ (which may depend on $\Omega$ and $T$) for which $A \le c B$.}

\subsection{Energy Balance} \label{ebal}

In order to present a weak formulation, we consider a formal energy balance. We allow  inertial and damping parameters to be present, later allowing them to vanish. At the end of the section, we will present the energy and dissipator for the quasi-static and undamped dynamics of central interest. 

Formally testing the system \eqref{biot1**}--\eqref{stokes1} with $(\vec{u}_t, p_b, \vec{v})$, resp., zeroing out  sources, and assuming the solution is smooth, we invoke the relevant boundary and coupling conditions to obtain:
\begin{align*}
0=&~    \rho_b(\vec{u}_{tt},\vec{u}_t)_{\Omega_b} - (\text{div}~\sigma^E(\vec{u}+\delta \bu_t),\vec{u}_t)_{\Omega_b} + \alpha (\nabla p_b, \vec{u}_t)_{\Omega_b} + c_0((p_{b})_t,p_b)_{\Omega_b}  \\[.1cm]
    &+ \alpha (\nabla\cdot \vec{u}_t,p_b)_{\Omega_b}- (\nabla\cdot [k\nabla p_b],p_b)_{\Omega_b}+ \rho_f(\vec{v}_t,\vec{v})_{\Omega_f} -2\nu (\nabla\cdot \mathbf D(\bv),\vec{v})_{\Omega_f} + (\vec{v},\nabla p_f)_{\Omega_f}\\[.2cm]
    =&~ \frac{1}{2}\frac{d}{dt}\Big[\rho_b\|\vec{u}_t\|_{0,b}^2 + \|\vec{u}\|_E^2 + c_0\|p_b\|_{0,b}^2 + \rho_f\|\vec{v}\|_{0,f}^2\Big]  + k\|\nabla p_b\|_{0,b}^2 + 2\nu\|\vec{D}(\vec{v})\|_{0,f}^2  \\ 
    &+ \beta \|(\vec{v}-\vec{u}_t)\cdot\boldsymbol{\tau} \|_{\Gamma_I}^2+\delta ||\bu_t||_E^2,
\end{align*}
Above, we  invoked the  stress-strain  {inner product} \begin{equation}\label{korny} \|\vec{u}\|^2_E \equiv  (\sigma^E(\vec{u}),\vec{D}(\vec{u}))_{\mathbf L^2(\Omega)},\end{equation} which is equivalent to the standard $\mathbf H^1(\Omega_b)$ norm here \cite{show2000,AGW}. Then, we can define a total energy $e^{\rho}(t)$  and  a dissipator $d^{\delta}(t)$, denoting dependencies on inertial and damping parameters:  \begin{align} \label{energy} e^{\rho}(t) \equiv &~ \frac{1}{2}\Big[\rho_b\|\vec{u}_t\|_{0,b}^2 + \|\vec{u}\|_E^2 + c_0\|p_b\|_{0,b}^2 + \rho_f\|\vec{v}\|_{0,f}^2\Big],\\
\label{dissipator} d^{\delta}(t) \equiv&~ \int_0^t\big[\delta ||\bu_t||_E^2+k\|\nabla p_b\|_{0,b}^2 
+ 2\nu \|\vec{D}(\vec{v})\|_{0,f}^2 
+ \beta \|(\vec{v}-\vec{u}_t)\cdot\boldsymbol{\tau} \|_{\Gamma_I}^2 \big] d\tau. \end{align}
From this, we obtain the formal energy balance for all $\delta,\rho_b,\rho_f,c_0 > 0$:
\begin{equation} \label{eident} e^{\rho}(t) +d^{\delta}(t) = e^{\rho}(0).\end{equation}

This identity holds for sufficiently smooth solutions, and, can be extended to weak solutions for finite energy data. We  refer to the {\em energy inequality} as
\begin{align} \label{eidentt}
e^{\rho}(t) + d^{\delta}(t) \le & ~ e^{\rho}(0), \\ \label{eidenttt}
e^{\rho}(t) + d^{\delta}(t) \lesssim & ~ e^{\rho}(0) + \int_0^t\big[||\mathbf F_b(\tau)||^2_{\vec{L}^2(\Omega_b)}+||\mathbf F_f(\tau)||^2_{[\mathbf H^1_{\#,*}(\Omega_f)\cap \mathbf V]'}+||S(\tau)||_{[H^1_{\#,*}(\Omega_b)]'}^2 \big]d\tau,
\end{align}
in the first case when $\mathbf F_f = \mathbf F_b = S \equiv 0$, and in the latter case when sources are present; relevant function spaces are defined  in Section \ref{absModel}. 
The primary case of interest is when  $\delta=0$ and $\rho_b=\rho_f=0$, yielding $e^0(t) \equiv e(t)$ and $d^{0}(t) \equiv d(t)$: \begin{align} e(t) =  \frac{1}{2}\Big[  \|\vec{u}\|_E^2 + c_0\|p_b\|_{0,b}^2 \Big],~~ d(t) = \int_0^t\big[k\|\nabla p_b\|_{0,b}^2 
+ 2\nu \|\vec{D}(\vec{v})\|_{0,f}^2 
+ \beta \|(\vec{v}-\vec{u}_t)\cdot\boldsymbol{\tau} \|_{\Gamma_I}^2 \big] d\tau.\end{align}

\begin{remark}[Initial States] \label{initialstates} We note that the initial conditions appearing in the inertial energy $e^{\rho}(0)$ are of the form 
$$\bu(0), ~~~~\rho_b\bu_t(0),~~~~c_0p_b(0),~~~~\rho_f\bv(0),$$
and hence only $\bu(0)$ and $c_0p_b(0)$ survive the inertial limit as $\rho_f,\rho_b \searrow 0$. We remark that our theorems are stated below for initial conditions of the form 
$\bu(0)=\bu_0$ and for the fluid content $[c_0p_b+\alpha \nabla\cdot \bu](0)=d_0$; when $c_0>0$, the quantity $p_b(0)$ can be recovered from 
$$p_b(0) = c_0^{-1}d_0-\alpha c_0^{-1}\nabla \cdot \bu_0,$$ and when $c_0=0$, we have the implicit compatibility condition that $d_0 = \alpha \nabla \cdot \bu_0$. 
\end{remark}
\subsection{Spaces}\label{absModel}
For either domain $\Omega_i$ ($i=b,f$), we introduce the notation $H^s_\#(\Omega_i)$ for the space of  functions from $H^s(\Omega_i)$ that are periodic in directions $x_1$ and $x_2$ (with spatial period 1). Accompanying these spaces, we have trace spaces $H_{\#}^{r}(\partial \Omega_f)$ for $r=\pm 1/2, \pm 3/2$, as defined in detail in \cite{AGW}.
Now let
\begin{align*}
    \vec{U} &\equiv \{\vec{u} \in \vec{H}_\#^1(\Omega_b)~:~~ \vec{u}\big|_{\Gamma_b} = \vec{0}\},\\
    \vec{V} &\equiv \{\vec{v} \in \vec{L}^2(\Omega_f) ~:~~ \text{div}~ \vec{v} \equiv 0 \text{ in }\Omega_f; \, ~~ \vec{v}\cdot\vec{n} \equiv 0\text{ on } \Gamma_f\},\end{align*}
    and define the {\em state space} for the inertial filtration by
    \begin{align}
    X &\equiv  \vec{U} \times \vec{L}^2(\Omega_b) \times L^2(\Omega_b) \times \vec{V}.
\end{align}
We will also use the notation below, $\vec{H}^1_{\#,\ast}(\Omega_f)$ and ${H}^1_{\#,\ast}(\Omega_b)$, where the subscript $*$ denotes a zero Dirichlet trace on $\Gamma_f$ and $\Gamma_b$, resp., that is, on the $\{x_3=-1\}$ or $\{x_3=1\}$ faces.
We will {subsequently} topologize the space $H^1_{\#,*}(\Omega_b)$ with the standard gradient norm, via the Poincar\'e inequality \cite{kesavan,bgsw}: $$||.||_{H^1_{\#,*}(\Omega_b)} \equiv ||\nabla ~.||_{L^2(\Omega_b)}.$$
The topology of $\mathbf U$ is induced by  $||\cdot||_E$, introduced before in \eqref{korny} through the bilinear form \begin{equation}\label{bilform} a_E(\cdot,\cdot)=(\sigma^E(\cdot),\mathbf D(\cdot))_{\mathbf L^2(\Omega_b)}.\end{equation} As noted, the induced norm is equivalent to the $\mathbf H^1(\Omega_b)$ norm on $\mathbf U$. 

When $c_0,\rho_b,\rho_f>0$, we introduce equivalent topologies through a particular inner-product on the energy space $X$:
\begin{align}\label{innerX}
(\mathbf y_1, \mathbf y_2)_X = a_E(\bu_1,\bu_2)+\rho_b(\bw_1,\bw_2)_{\mathbf L^2(\Omega_b)}+c_0(p_1,p_2)_{L^2(\Omega_b)}+\rho_f(\mathbf v_1,\mathbf v_2)_{\mathbf L^2(\Omega_f)},
\end{align}
where the state variable $\vec{w}$ captures the Biot elastic velocity $\vec{u}_t$.

\subsection{Definition of Weak Solutions}
{We will now define weak solutions for all $\rho_b, \rho_f, c_0, \delta \ge 0$---the single definition holding across all cases.} Our function spaces will then depend on these parameter values, but we follow the convention of \cite{multilayered} and encode the dependencies into the spaces, allowing us to omit superscripts in our descriptions. 
To that end, define the spaces
\begin{align*}
    \mathcal{V}_b &= \{\vec{u} \in L^\infty(0,T;\vec{U}) ~:~ \rho_b\vec{u} \in W^{1,\infty}(0,T;\vec{L}^2(\Omega_b)),~\delta \bu_t \in L^2(0,T;\mathbf U)\}, \\
    \mathcal{Q}_b &= \left\{p \in L^2(0,T; H^1_{\#,*}(\Omega_b))~:~c_0p \in L^\infty(0,T;L^2(\Omega_b))\right\}, \\
    \mathcal{V}_f &= \left\{ \bv \in L^2(0,T; \vec{H}^1_{\#,\ast}(\Omega_f) \cap \vec{V})~:~\rho_{f} \bv \in L^\infty(0,T;\vec{V})\right\}.
\end{align*}
Then, the {\em weak solution space} is
\newcommand\Vsol{\mathcal{V}_{\text{sol}}}
\newcommand\Vtest{\mathcal{V}_{\text{test}}}    
\[
    \mathcal{V}_\text{sol} = \mathcal{V}_b\times \mathcal{Q}_b\times\mathcal{V}_f,~\text{ for } [\bu,p_b,\bv]^T
\]
and the {\em test space} is

\[
    \mathcal{V}_\text{test} = C^1_0\left([0,T); \vec{U} \times  H^1_{\#,*}(\Omega_b) \times (\vec{H}^1_{\#,\ast}(\Omega_f)\cap \vec{V})\right),~\text{ for } ~[\bxi,q_b,\bzeta]^T.
\]

To give the weak formulation, we use the time-space conventions:
 $((\cdot, \cdot))_{\mathscr O}$ denotes an inner-product on $L^2(0,T;L^2(\mathscr O))$, for a spatial domain $\mathscr O$. Similarly, for a Sobolev space defined on $\mathscr O$, say $W(\mathscr O)$, we will denote $(\langle \cdot,\cdot \rangle)_{\mathscr O}$  as pairing between $L^2(0,T;W(\mathscr O))$ and its dual $L^2(0,T;[W(\mathscr O)]')$.
 We will consider initial conditions of the form:
     \begin{equation}\label{weakICs}
        \vec{u}(0) = \vec{u}_0, ~~ \rho_{b}\vec{u}_t(0) = \rho_{b}\vec{u}_1, ~~ [c_0p_b + \alpha\nabla\cdot\vec{u}](0) = d_0, ~~ \rho_{f}\vec{v}(0) = \rho_{f}\vec{v}_0.
    \end{equation}
    Note that, again, the above convention allows us to treat the case of vanishing inertial parameters by reducing to tautologies in those cases. We note that an implicit compatibility condition on the data, as presented above, will apparently be required when $c_0=0$ such that $d_0= \alpha \nabla \cdot \bu_0$ in the appropriate sense---see Remark \ref{initialstates}. 

   Finally, we consider source data of regularity (at least): 
    $$\mathbf F_f \in L^2\left(0,T;(\vec{H}^1_{\#,\ast}(\Omega_f)\cap \vec{V})'\right),~~\mathbf F_b \in L^2\left(0,T;\vec{L}^2(\Omega_b)\right),~~\text{and}~~S \in L^2\left(0,T; (H^{1}_{\#,*}(\Omega_b))'\right).$$

\begin{definition}\label{weaksols}
    For any  $\rho_b,\rho_f,c_0,\delta \ge 0$, we  say that $[\vec{u},p,\vec{v}]^T \in \mathcal{V}_\text{sol}$ is a weak solution to the filtration: \\[.2cm] (1) 
    For every test function $[\bxi,q_b,\bzeta]^T \in \mathcal{V}_\text{test}$ the following identity holds:
    \begin{align}\label{weakform}
        -&~ \rho_b((\vec{u}_t,\bxi_t))_{\Omega_b} + ((\sigma_b(\vec{u},p), \nabla \bxi))_{\Omega_b}- ((c_0 p + \alpha\nabla\cdot\vec{u}, \partial_t q_b))_{\Omega_b} + ((k\nabla p,\nabla q_b))_{\Omega_b} \nn\\
        &- \rho_f((\vec{v},\bzeta_t))_{\Omega_f} + 2\nu((\vec{D}(\vec{v}),\vec{D}(\bzeta)))_{\Omega_f} + (({p,(\bzeta-\bxi)\cdot\vec{e}_3}))_{\Gamma_I} - (({\vec{v}\cdot\vec{e}_3,q_b}))_{\Gamma_I} \nn\\
        &- (({\vec{u}\cdot\vec{e}_3,\partial_tq_b}))_{\Gamma_I} + \beta (({\vec{v}\cdot\btau,(\bzeta - \bxi)\cdot\btau}))_{\Gamma_I} + \beta(({\vec{u}\cdot\btau, (\bzeta_t - \bxi_t)\cdot\btau}))_{\Gamma_I} \nn\\
        =&~ \rho_b(\mathbf u_1,\bxi)_{\Omega_b}\big|_{t=0} + (d_0,q_b)_{\Omega_b}\big|_{t=0} + \rho_f(\vec{v}_0,\bzeta)_{\Omega_f}\big|_{t=0} + ({\vec{u}_0\cdot\vec{e}_3,q_b})_{\Gamma_I}\big|_{t=0} \nn\\
        &~- \beta({\vec{u}_0\cdot\btau, (\bzeta - \bxi)\cdot\btau})_{\Gamma_I}\big|_{t=0} + ((\vec{F}_b,\bxi))_{\Omega_b} + (\langle S,q_b\rangle)_{\Omega_b} + (\langle \vec{F}_f,\bzeta\rangle)_{\Omega_f},
    \end{align}
 for $\btau=\mathbf e_i,~i=1,2$.
    \\[.2cm] (2) The function $\mathbf u \in \mathcal V_b$ has the additional property that $\gamma_0[\mathbf u]_t \cdot \btau \in L^2(0,T;L^2(\Gamma_I))$ for $\btau=\mathbf e_1,\mathbf e_2$.
\end{definition}
\begin{remark} As in previous works such as \cite{multilayered,AGW,AW} the weak formulation carries the initial conditions on the right hand side. Thus we do not include the initial conditions explicitly as separate criteria in the definition of weak solutions. This is a subtle but important point, since the time regularity of solutions is at issue in degenerate models. Later works will investigate time continuity of Biot displacements, along the lines of \cite{bmw}.  \end{remark}

 The approach to constructing solutions in this paper is based on considering an underlying  $C_0$-semigroup \cite{pazy} for the non-degenerate and damped dynamics---$\rho_f,\rho_b,c_0,\delta>0$. The main result for this case will be summarized in Section \ref{subsection:main_result}. A detailed discussion of the semigroup and its generator will be given in the Appendix. We note here that the semigroup provides strong solutions for appropriately smooth data. 

\section{Main Results and Review of Supporting Literature} \label{section:main_result}

In this section, we present our main existence result, Theorem \ref{th:main1}, for quasi-static weak solutions, followed by some auxiliary results; the latter results serve as the foundation for the former.

\subsection{Main Result and Discussion}\label{subsection:main_result}
Our main result, Theorem \ref{th:main1}, covers the existence of weak solutions in the sense of Definition \ref{weaksols} with $\delta=\rho_f=\rho_b=0$.

 \begin{theorem}\label{th:main1} Suppose $\mathbf F_f \in L^2(0,T;[\vec{H}^1_{\#,\ast}(\Omega_f)\cap \vec{V}]')$, $\mathbf F_b \in L^2(0,T;\mathbf L^2(\Omega_b))$, and $S \in L^2(0,T; [H^{1}_{\#,*}(\Omega_b)]')$.
   Then for $\rho_b=\rho_f=\delta =0$ and any $c_0\ge 0$, the dynamics admit a weak solution corresponding to the initial conditions
\begin{align} 
    \bu(0)=\bu_0 \in \vec{U},~~ [c_0p+\alpha \nabla \cdot \bu](0)=d_0 \in L^2(\Omega_b).
\end{align}
(When $c_0=0$, we require $\alpha\nabla \cdot \bu_0=d_0 \in L^2(\Omega_b).$)

   This weak solution satisfies the energy inequality in \eqref{eidenttt}. 
\end{theorem}

The strategy for proving our main theorem is as follows. We begin by invoking the well-posedness of the system \eqref{biot1**}--\eqref{IC4*} for $c_0$, $\rho_b,\rho_f,$ and $\delta>0$. This result follows from the existence of an underlying $C_0$-semigroup \cite{pazy} associated with the Biot-Stokes dynamics, whose existence can be established in a manner similar to that found in \cite{AGW}---see Section \ref{subsection:Ancillary_Result} below. A detailed discussion of the semigroup and its generator is provided in the Appendix. Next, we send these parameters to zero in the following order: the inertial parameters $\rho_b$ and $\rho_f$ are sent to zero with $\delta>0$ fixed (Section \ref{subsection:vanished_inertia}), and then $\delta$ is sent to zero in a latter step (Section \ref{subsection:vanishing_damping}). The proof below is written with $c_0>0$; the storage-degenerate case $c_0=0$ is not shown here, since it is achieved by the same finite-energy limiting mechanism used in \cite{AGW,bw} and is unaffected by the treatment of inertia/viscoelasticity here.

We now frame the preceding theorem in relation to the closest existing results. Showalter's filtration work \cite{showfiltration} is an important antecedent of the present work: it couples poroelasticity to Stokes flow and treats the problem by implicit, degenerate methods \cite{indiana,show2000}. That work, however, is formulated for a compressible free flow and inertial poroelasticity; its proof does not carry  limiting constructions for   quasi-static Biot-Stokes dynamics with free-flow incompressibility. Cesmelioglu's analysis of the coupled Navier-Stokes-Biot problem \cite{filtration2} is relevant; there, the fluid-poroelastic system is inertial and the fluid nonlinear. The result gives existence and uniqueness under hypotheses appropriate to the Navier-Stokes-Biot coupling, with stronger data requirements than that of finite energy. 

On the numerical-analysis side, works by Yotov et al. (such as \cite{yotovLagrange, yotov2022,yotovMultipoint,yotovNSB}) develop  Lagrange-multiplier, mixed, multipoint stress-flux, and augmented  mixed formulations for Stokes-Biot or Navier-Stokes-Biot systems. Much of this work develops continuous formulations  in concert with semidiscrete or fully discrete error analysis. While these works are close to this analysis,  their focus is on robust discretization, local conservation, and mixed-variable formulations. In particular, continuous weak theories are  posed in non-degenerate storage regimes. Moreover, in these formulations, weak solutions require data with  compatibility or additional regularity conditions tailored to the  formulation; we contrast with the result here, where weak solutions are constructed from finite-energy data---in the sense of Definition~\ref{weaksols}---utilizing the regularity needed for the natural energy estimates and including the   $c_0=0$ case. 

The issue of uniqueness of energy-level weak solutions is closely tied to regularity in this class of problems: the inertial analyses \cite{AGW,AW} show the role played by interface and time regularity in the full 3D-3D coupling, while even quasi-static Biot dynamics alone exhibit subtle time-regularity effects \cite{bw,bmw,bmw2}. For this reason, the present paper focuses on the construction of quasi-static weak solutions in this previously open regime, leaving refined regularity and uniqueness questions to subsequent work. Thus the point of the present construction is not to characterize all possible weak solutions, but to obtain finite-energy weak solutions in a singular quasi-static regime where direct energy-level constructions are unavailable.

\subsection{Ancillary Results: Semigroup Solutions} \label{subsection:Ancillary_Result}

For the non-degenerate Biot system, where $\rho_b$, $\rho_f$, $\delta$, and $c_0$ are positive, the dynamics 
\eqref{biot1**}--\eqref{IC4*} can be represented by an operator 
$\cA_{\delta}$ acting on the state space
\[
X \equiv \vec{U} \times \vec{L}^2(\Omega_b) \times L^2(\Omega_b) \times \vec{V}.
\]
To describe this operator, we first introduce the principal differential 
operators associated with the Biot component:
\[
\cE_0 = -\rho_b^{-1}\nabla\cdot\sigma^E,
\qquad
A_0 = -c_0^{-1}k\Delta.
\]
With this notation, the dynamics operator is given by
\begin{equation}\label{opfirst}
    \cA_{\delta} \equiv 
    \begin{pmatrix}
        0 & \vec{I} & 0 & 0\\
        -\cE_0 & -\delta \cE_{0} & -\alpha\rho_b^{-1} \nabla & 0\\
        0 & -\alpha c_0^{-1} \nabla\cdot & -A_0 & 0\\
        0 & 0 & \rho_f^{-1}G_1 & 
        \rho_f^{-1}[\nu\Delta + G_2 + G_3]
    \end{pmatrix}.
\end{equation}The operators $G_1, G_2,$ and $G_3$ above are certain Green's maps which enable the elimination of the pressure (and carry boundary information) \cite{avalos*,AGW,AW}. The operator $\cA_{\delta}$ will be discussed in detail in the Appendix. The result below is obtained identically as in \cite{AGW} for the inertial $\rho_b, \rho_f>0$ but undamped filtration $\delta=0$, simply accounting for the addition of strong damping in the hyperbolic component of the dynamics.

\begin{theorem}
\label{th}
    The operator $\mathcal A_{\delta}: \mathcal D(\mathcal A_{\delta})\subset X \to X$ given in Definition \ref{diffdomain} is the generator of a strongly continuous semigroup $\{e^{\cA_{\delta} t}: t\geq 0\}$ of contractions on $X$. Thus, for $\mathbf y_0 \in \mathcal D(\mathcal A_{\delta})$, we have $e^{\mathcal A_{\delta}\cdot}\mathbf y_0 \in C([0,T];\mathcal D(\mathcal A_{\delta}))\cap C^1((0,T);X)$ satisfying \eqref{cauchy} in the {\em strong sense} with $\mathcal F=[\mathbf 0,\mathbf 0, 0,\mathbf 0]^T$. 
    
    Similarly, for $\mathbf y_0 \in X$, we have $e^{\mathcal A_{\delta}\cdot}\mathbf y_0 \in C([0,T];X)$ satisfying \eqref{cauchy} in the {\em generalized} or semigroup sense  with $\mathcal F=[\mathbf 0, \mathbf 0, 0,\mathbf 0]^T$.
\end{theorem}
As a corollary, we will obtain strong and weak solutions to \eqref{biot1**}--\eqref{IC4*}. By strong solutions to  \eqref{biot1**}--\eqref{IC4*} we mean weak solutions, as in Definition \ref{weaksols}, with enough regularity that the equations in \eqref{biot1**}--\eqref{IC4*} also hold pointwise. For the results presented below, we do not explicitly mention generalized solutions, as these are tied specifically to the semigroup framework (though the well-posedness of generalized solutions also follows from Theorem \ref{th}).

\begin{corollary}\label{coro:regular_strong_weak_existence}
\hfill
\begin{enumerate}
 \item[(1)] Take $\mathbf F_f\equiv \mathbf 0,$ $\mathbf F_b\equiv \mathbf 0$, and $S\equiv 0$. 

a) Suppose $\mathbf y_0 \in \mathcal D(\mathcal A_{\delta})$. Then strong solutions to \eqref{biot1**}--\eqref{IC4*}  exist and are unique. For such strong solutions, the energy relation \eqref{eident} holds.

b) Suppose $\mathbf y_0 \in X$. Then weak solutions as defined in Definition \ref{weaksols} exist. For such weak solutions, the energy inequality \eqref{eidentt} holds.

 \item[(2)] Suppose $\mathbf y_0 \in \mathcal D(\mathcal A_{\delta})$ and $\mathcal F = [\mathbf 0,\mathbf  F_b, S, \mathbf F_f]^T \in H^1(0,T;X)$. Then strong solutions to \eqref{cauchy} exist and are unique; the constructed solutions satisfy the energy inequality \eqref{eidenttt}.
 
\item[(3)] Suppose $\mathbf y_0 \in X$ and $\mathbf F_f \in L^2\big(0,T;(\vec{H}^1_{\#,\ast}(\Omega_f)\cap \vec{V})\big)'$, $\mathbf F_b \in L^2(0,T;\vec{L}^2(\Omega_b))$, and $S \in L^2(0,T; [H^{1}_{\#,*}(\Omega_b)]')$. Then weak solutions to \eqref{cauchy} exist; the constructed solutions satisfy the energy inequality \eqref{eidenttt}.
\end{enumerate}
\end{corollary}

The statements about {\em strong solutions} in Corollary \ref{coro:regular_strong_weak_existence}---(1a) and (2)---follow immediately from the standard theory of semigroups. Namely, when $\mathbf y_0 \in \mathcal D(\cA_{\delta})$ (and $\mathcal F=[\mathbf 0, \mathbf F_b,S,\mathbf F_f]^T \in H^1(0,T;X)$ for the inhomogeneous problem) we observe that the function $\mathbf y(t) =e^{\mathcal A_{\delta} t}\mathbf y_0$ is the unique solution (pointwise in time) to the Cauchy problem
$\dot{\mathbf y} = \mathcal A_{\delta}\mathbf y+\mathcal F$~ with $\mathbf y(0)=\mathbf y_0$; in addition, $e^{\mathcal A_{\delta} t}\mathbf y_0 \in C^1((0,T);X)\cap C([0,T];\mathcal D(\cA_{\delta}))$. From the definition of  $\mathcal D(\cA_{\delta})$, the equations in \eqref{biot1**}--\eqref{stokes1} hold pointwise in time, $a.e.$ in $\mathbf x$.  Moreover, the boundary conditions in \eqref{BC1} and \eqref{IC1*}--\eqref{IC4*} hold pointwise in time, in the sense of the definition $\mathcal D(\cA_{\delta})$ in Definition \ref{diffdomain}.

In Section \ref{section:proof_of_main_theorem}, we prove the existence of weak solutions for \eqref{biot1**}--\eqref{IC4*} for the case of zero inertia parameters $\rho_b$, $\rho_f$ and zero damping coefficient $\delta$. The storage-degenerate case $c_0=0$ is handled as above through the finite-energy limiting argument from \cite{AGW,AW}. \\

\section{Proof of Theorem~\ref{th:main1}} \label{section:proof_of_main_theorem}
We proceed in steps. We first construct weak solutions with $\delta,c_0>0$ sending the inertial parameters to zero  in Section \ref{subsection:vanished_inertia}. Then we send $\delta \to 0$  in Section \ref{subsection:vanishing_damping}. Finally, we remark that $c_0$ can be sent to zero in the by now standard manner. 

 \subsection{Vanishing Inertia: $\rho_b$, $ \rho_f \to 0$} \label{subsection:vanished_inertia}
The first step in proving Theorem~\ref{th:main1} is to prove the existence of a weak solution for $\rho_b=\rho_f = 0$. In this intermediate step, we fix $\delta>0$. This allows us to control the quantity $\bu_t$ independent of the parameter $\rho_b$.  Also, in this section we hold $c_0>0$. Under these assumptions, the filtration system reduces to
\begin{align}\label{introsystem3}
- \mu\Delta(\vec{u}+\delta\bu_t) - (\lambda + \mu)\nabla(\nabla\cdot(\vec{u}+\delta \bu_t)) + \alpha \nabla p_b = \vec{F}_b, &\text{ in } \Omega_b \times (0,T),\\ 
    [c_0p_b + \alpha \nabla \cdot \vec{u}]_t - \nabla \cdot [k\nabla p_b] = S, &\text{ in } \Omega_b \times (0,T)\\
  -\nu\Delta \bv +\nabla p_f=\mathbf F_f;~~\nabla \cdot \bv = 0, & \text{ in } \Omega_f \times (0,T).
\end{align}
Since the limit system is now first order in time, the requisite initial data is 
$$\bu(0)=\bu_0 \in \vec{U},~~ [c_0p+\alpha\nabla \cdot \bu](0)=d_0 \in L^2(\Omega_b).$$

We denote $\vec{\rho} \equiv (\rho_b, \rho_f)$. As we are performing a construction of solutions, when writing $\vec{\rho} \rightarrow 0$ we abuse the notation and mean to choose a sequence (perhaps refining to a subsequence) so that $\rho = \left((\rho_b)_n, (\rho_f)_n\right)=(\rho_n,\rho_n) \to (0,0)$.

We will obtain an auxiliary result for weak solutions by using the semigroup solutions coming from Theorem \ref{th}.

\begin{theorem} \label{theorem:vanished_inertia} Suppose $d_0 \in L^2(\Omega_b),~\bu_0 \in \mathbf U$ with $\mathbf F_f \in L^2(0,T;[\vec{H}^1_{\#,\ast}(\Omega_f)\cap \vec{V}]')$, $\mathbf F_b \in L^2(0,T;\mathbf L^2(\Omega_b))$, and $S \in L^2(0,T; [H^{1}_{\#,*}(\Omega_b)]')$. 
   Then for any $\delta, c_0> 0$, with $\rho_b=\rho_f = 0$, the dynamics in \eqref{introsystem3} admit a weak solution $(\overline{\bu}, \overline{p}, \overline{\bv})^T \in \mathcal{V}_{sol}$. The constructed weak solution  satisfies the energy inequality in \eqref{eidenttt} (taking $\rho_b = \rho_f = 0$). Furthermore,  $ \overline{\mathbf u} \in \mathcal V_b$ has the property that $\gamma_0[\overline{\mathbf u}]_t \cdot \btau \in L^2(0,T;L^2(\Gamma_I))$ for $\btau=\mathbf e_1,\mathbf e_2$.
\end{theorem}

\begin{remark} When $\rho_b = \rho_f = 0$, the constituent spaces reduce to
\begin{align*}
    \mathcal{V}_b &= \{\vec{u} \in L^\infty(0,T;\vec{U}) ~:~\delta \bu_t \in L^2(0,T;\mathbf U)\}, \\
    \mathcal{Q}_b &= \left\{p \in L^2(0,T; H^1_{\#,*}(\Omega_b))~:~c_0p \in L^\infty(0,T;L^2(\Omega_b))\right\}, \\
    \mathcal{V}_f &= \left\{ \bv \in L^2(0,T; \vec{H}^1_{\#,\ast}(\Omega_f) \cap \vec{V})~\right\}.
\end{align*}

\end{remark}
\begin{proof}[Proof of Theorem \ref{theorem:vanished_inertia}]
Fix any arbitrary  $\bu_1 \in \vec{L}^2(\Om_b)$ and $\bv_0 \in \bV$. Let $\vec{y}_0 = [\vec{u}_0,\vec{u}_1,p_0,\vec{v}_0]^T\in X$. For any $\rho_b,~\rho_f > 0$ and a fixed initial condition $\vec{y}_0$, Corollary \ref{coro:regular_strong_weak_existence} (3) guarantees the existence of a weak solution $(\bu^{\rho},p^{\rho},\bv^{\rho})$. Furthermore,  $(\bu^{\rho},p^{\rho},\bv^{\rho})$ satisfies the following energy estimate
\begin{align}
    e^{\rho}(t) + d^{\delta}(t) \lesssim & ~ e^{\rho}(0) + \int_0^t\big[||\mathbf F_b(\tau)||^2_{\vec{L}^2(\Omega_b)}+||\mathbf F_f(\tau)||^2_{[\mathbf H^1_{\#,*}(\Omega_f)\cap \mathbf V]'}+||S(\tau)||_{[H^1_{\#,*}(\Omega_b)]'}^2 \big]d\tau,
\end{align}
where we recall
\begin{align*} e^{\rho}(t) \equiv &~ \frac{1}{2}\Big[\rho_b\|\vec{u}_t^{\rho}\|_{0,b}^2 + \|\vec{u}^{\rho}\|_E^2 + c_0\|p^{\rho}\|_{0,b}^2 + \rho_f\|\vec{v}^{\rho}\|_{0,f}^2\Big],\ \\ 
d^{\delta}(t) \equiv&~ \int_0^t\big[\delta ||\bu_t^{\rho}||_E^2+k\|\nabla p^{\rho}\|_{0,b}^2 
+ 2\nu \|\vec{D}(\vec{v}^{\rho})\|_{0,f}^2 
+ \beta \|(\vec{v}^{\rho}-\vec{u}^{\rho}_t)\cdot\boldsymbol{\tau} \|_{\Gamma_I}^2 \big] d\tau,\end{align*}
and we note that, since the initial conditions are fixed, $e^{\vec{\rho}_1}(0) \le e^{\vec{\rho}_2}(0)$ for all $0<\rho_1<\rho_2$. The upper bound in the energy inequality is therefore uniform in $\vec{\rho}$.

We now note that the quantity $\{\rho_b^{1/2}\bu_t^{\rho}\}$ is bounded in $L^{\infty}(0,T;\vec{L}^2(\Omega_b))$ (from the energy estimate). Thus this sequence has a weak-$*$ subsequential limit. Additionally, restricting to that subsequence, and relabeling, we next note that the dissipator $d^{\delta}$ yields that $\bu_t^{\rho}$ is bounded in $L^2(0,T; \vec{U})$. Therefore the sequence also has a weak subsequential limit in the latter sense, which coincides with $\overline \bu_t$. Again, restricting to the further subsequence and relabeling, we have a coincident limit for the sequence $\rho_b^{1/2}\bu_t^{\rho}$ in $L^{\infty}(0,T;\vec{L}^2(\Omega_b))\cap L^2(0,T;\vec{U})$.  On the other hand, since $\rho_b \searrow 0$ and  $\{\rho_b^{1/2}\bu_t^{\rho}\}$ is bounded in $L^2(0,T;\vec{U})$, it must be the case that $\{\rho_b^{1/2}\bu_t^{\rho}\}$ goes to zero in both senses. By a similar argument, we also have the weak-$*$ convergence of $\{\rho_f^{1/2}\bv^{\rho}\}$ to $\vec{0}$ in $L^\infty (0,T;\vec{L}^2(\Omega_f))$ and the weak convergence of $\{\bv^\rho\}$ in $\rho$ to some limit $\overline{\bv}$ in $L^2(0,T;H^1_{\#,*}(\Omega_f))$.

By Banach-Alaoglu, for the sequence $\{\bu^{\rho},p^{\rho}\}$, we thus have the following weak-$*$ (subsequential) limits and limit points, upon relabeling:
\begin{align*}
\bu^{\rho} \rightharpoonup^* &~\overline{\bu} \in L^{\infty}(0,T;\mathbf U)\\
p^{\rho}  \rightharpoonup^* &~\overline{p} \in L^{\infty}(0,T;L^2(\Omega_b)).
\end{align*}
From the dissipator, we  have the weak limits, which are identified as above (by uniqueness):
\begin{align*}
\bu^{\rho}_t \rightharpoonup &~\overline{\bu}_t  \in L^{2}(0,T;\vec{U})\\
\bv^{\rho} \rightharpoonup &~\overline{\bv} \in L^{2}(0,T;\mathbf  H^1_{\#,*}(\Omega_f))\\
p^{\rho}  \rightharpoonup &~\overline{p} \in L^{2}(0,T;H^1_{\#,*}(\Omega_b))\\
(\bv^{\rho}-\bu_t^{\rho})\cdot \btau \rightharpoonup &~ (\overline{\mathbf v} - \overline{\bu}_t ) \cdot \btau \in L^2(0,T;L^2(\Gamma_I)).
\end{align*}

We now consider the terms in the weak formulation \eqref{weaksols} individually, and justify their convergence. We recall that
   $$ [\bxi,q,\bzeta]^T \in \mathcal{V}_\text{test} = C^1_0([0,T); \vec{U} \times  H^1_{\#,*}(\Omega_b) \times (\vec{H}^1_{\#,\ast}(\Omega_f)\cap \vec{V})).$$

Since $p^{\rho}  \rightharpoonup ~\overline{p} \in L^{2}(0,T;H^1_{\#,*}(\Omega_b))$ and $\bu^{\rho} \rightharpoonup^* ~\overline{\bu} \in L^{\infty}(0,T;\mathbf U)$, we can infer that 
\begin{equation}\label{weake1}   \sigma^E(\bu^\rho)- \alpha p^\rho\vec{I}\rightharpoonup \sigma^E(\overline{\vec{u}}) - \alpha \overline{p} \vec{I} \in  L^2(0,T; [L^2(\Omega_b)]^{3\times 3}),\end{equation} noting the embedding $L^2(0,T;Z) \hookrightarrow L^1(0,T;Z)$. Combining this with the fact that  $\bu^{\rho}_t \rightharpoonup ~\overline{\bu}_t \in L^{2}(0,T;\vec{U})$, we have
\begin{align}
    ((\sigma_b(\bu^\rho,p^\rho), \nabla \bxi ))_{\Omega_b} \longrightarrow~ &~((\sigma^E(\overline{\bu}), \nabla \bxi ))_{\Omega_b} -  \alpha ((\overline{p} \vec{I} , \nabla \bxi))_{\Omega_b} +\delta((\sigma^E(\overline{\bu}_t) , \nabla \bxi))_{\Omega_b}\\
    &~ \equiv ((\sigma_b(\overline{\bu},\overline p), \nabla \bxi ))_{\Omega_b}. 
\end{align}
From $\bv^{\rho} \rightharpoonup ~\overline{\bv} \in L^{2}(0,T;\mathbf \mathbf H^1_{\#,*}(\Omega_f))$, it is immediate that
\begin{equation}\label{weake2}\mathbf D(\bv^{\rho}) \rightharpoonup \mathbf D(\overline{\bv}) \in L^2(0,T;[L^2(\Omega_f)]^{3\times 3}).\end{equation}
Therefore, with $\mathbf D(\bzeta) \in L^2(0,T; [L^2(\Omega_f)]^{3\times 3})$,
\begin{align}
     2\nu((\vec{D}(\vec{v}^{\rho}),\vec{D}(\bzeta)))_{\Omega_f}\to&~2\nu((\vec{D}(\overline{\vec{v}}),\vec{D}(\bzeta)))_{\Omega_f}
\end{align}
Next, since $\bxi_t \in L^1(0,T;\mathbf L^2(\Omega_b))$ and $\bu^{\rho}_t\rightharpoonup^* \overline{\bu}_t \in L^{\infty}(0,T;\mathbf L^2(\Omega_b))$, 
\begin{align}
((\vec{u}_t^{\rho},\bxi_t))_{\Omega_b}  \to&~ ((\overline{\bu}_t,\bxi_t))_{\Omega_b} 
\end{align}
and therefore
\begin{align}
\rho_b((\vec{u}_t^{\rho},\bxi_t))_{\Omega_b}  \to&~ 0. 
\end{align}

Similarly, the following convergences hold
       \begin{align*}
     ((c_0 p^{\rho} + \alpha\nabla\cdot\vec{u}^{\rho}, \partial_t q))_{\Omega_b} \to &~    ((c_0 \overline p + \alpha\nabla\cdot \overline{\vec{u}}, \partial_t q))_{\Omega_b}\\ 
       \text{since} &  \text{  $q_t \in L^1(0,T; L^2(\Omega_b))$} \\ \text{and} & \text{ $c_0p^{\rho}+\nabla\cdot \bu^{\rho}\rightharpoonup^* c_0\overline p +\nabla \cdot \overline{\bu} \in L^{\infty}(0,T; L^2(\Omega_b))$}  \\[.2cm]
     ((k\nabla p^{\rho},\nabla q))_{\Omega_b} \to & ~ ((k\nabla \overline p, \nabla q))_{\Omega_b} \\
     \text{since} &  \text{  $\nabla q \in L^2(0,T; \mathbf L^2(\Omega_b))$} ~~ \text{and $\nabla p^{\rho} \rightharpoonup \nabla \overline p \in L^{2}(0,T; \mathbf L^2(\Omega_b))$}  \\[.2cm]
     \rho_f((\vec{v}^{\rho},\bzeta_t))_{\Omega_f} \to &~  
     0 \\ 
            \text{since} & \text{ $\bzeta_t \in L^1(0,T; \mathbf L^2(\Omega_f))$ and $\rho_f\bv^{\rho} \rightharpoonup^* \vec{0} \in L^{\infty}(0,T;\mathbf L^2(\Omega_f))$}
   \end{align*}

Now we consider the trace terms. We bear in mind the standard trace theorem on $H^1$-type spaces, i.e., $\mathbf {H}_\#^1(\Omega_i) \xrightarrow{\gamma_0} \vec{H}^{1/2}(\partial \Omega_i) \hookrightarrow \vec{L}^2(\partial \Omega_i)$,  but we suppress  trace operators $\gamma_0[\cdot]$; we then have
   \begin{align*}
       ~  (({p^{\rho},(\bzeta-\bxi)\cdot\vec{e}_3}))_{\Gamma_I} \to&~  (({\overline p,(\bzeta-\bxi)\cdot\vec{e}_3}))_{\Gamma_I}  \\ 
    \text{since }& \text{ $(\bzeta-\bxi)  \in L^2(0,T;\mathbf L^2(\Gamma_I))$} ~~ \text{and $p^{\rho} \rightharpoonup \overline{p} \in L^{2}(0,T;\mathbf L^2(\Gamma_I))$}  \\[.2cm]
     (({\vec{v}^{\rho}\cdot\vec{e}_3,q}))_{\Gamma_I} \to &~  ((\overline{\vec{v}}\cdot\vec{e}_3,q))_{\Gamma_I} \\
       \text{since }& \text{ $q \in L^2(0,T; L^2(\Gamma_I))$} ~~ \text{and $\bv^{\rho}\rightharpoonup \overline{\bv} \in L^{2}(0,T;\mathbf L^2(\Gamma_I))$}  \\[.2cm]
     (({\vec{u}^{\rho}\cdot\vec{e}_3,\partial_tq}))_{\Gamma_I} \to &~ (({\overline{\vec{u}}\cdot\vec{e}_3,\partial_tq}))_{\Gamma_I}\\
       \text{since }& \text{ $q_t \in L^1(0,T; L^2(\Gamma_I))$} ~~ \text{and $\bu^{\rho}\rightharpoonup^* \overline{\bu} \in L^{\infty}(0,T;\mathbf L^2(\Gamma_I))$} 
       \end{align*}
       \begin{align*}
     \beta (({\vec{v}^{\rho}\cdot\btau,(\bzeta - \bxi)\cdot\btau}))_{\Gamma_I} \to &~ \beta (({\overline{\vec{v}}\cdot\btau,(\bzeta - \bxi)\cdot\btau}))_{\Gamma_I} \\
      \text{since }& \text{ $(\bzeta-\bxi)  \in L^2(0,T;\mathbf L^2(\Gamma_I))$} ~~ \text{and  $\bv^{\rho}\rightharpoonup \overline{\bv} \in L^{2}(0,T;\mathbf L^2(\Gamma_I))$ }  \\[.2cm]
     \beta(({\vec{u}^{\rho}\cdot\btau, (\bzeta_t - \bxi_t)\cdot\btau}))_{\Gamma_I} \to &~ \beta(({\overline{\vec{u}}\cdot\btau, (\bzeta_t - \bxi_t)\cdot\btau}))_{\Gamma_I} \\
    \text{since }& \text{ $(\bzeta-\bxi)_t  \in L^1(0,T;\mathbf L^2(\Gamma_I))$} ~~~~ \text{and $\bu^{\rho}\rightharpoonup^* \overline{\bu} \in L^{\infty}(0,T;\mathbf L^2(\Gamma_I))$}  
   \end{align*}

Finally, since the same initial conditions apply for all values $\vec{\rho}$ by construction, 
\begin{align*}
&~ \rho_b(\bu_1,\bxi)_{\Omega_b}\big|_{t=0} + (c_0 p_0 + \alpha\nabla\cdot \bu_0,q)_{\Omega_b}\big|_{t=0} + \rho_f(\bv_0,\bzeta)_{\Omega_f}\big|_{t=0} + ({\bu_0\cdot\vec{e}_3,q})_{\Gamma_I}\big|_{t=0} \nn\\
    &~- \beta({\bu_0\cdot\btau, (\bzeta - \bxi)\cdot\btau})_{\Gamma_I}\big|_{t=0}+ ((\vec{F}_b,\bxi))_{\Omega_b} + (\langle S,q\rangle)_{\Omega_b} + (\langle \vec{F}_f,\bzeta\rangle)_{\Omega_f}
\end{align*}
converges to
\begin{align*}
    &~ (d_0,q)_{\Omega_b}\big|_{t=0} + ({\bu_0\cdot\vec{e}_3,q})_{\Gamma_I}\big|_{t=0} \nn\\
    &~- \beta({\bu_0\cdot\btau, (\bzeta - \bxi)\cdot\btau})_{\Gamma_I}\big|_{t=0}+ (\langle \vec{F}_b,\bxi\rangle)_{\Omega_b} + (\langle S,q\rangle )_{\Omega_b} + (\langle \vec{F}_f,\bzeta\rangle)_{\Omega_f}.
\end{align*}

Thus, as $\delta$ is fixed, we may pass to the limit as $\vec\rho \to 0$ in the weak formulation to obtain for all $[\bxi,q,\bzeta]^T \in \mathcal V_{\text{test}}$:
\begin{align}\label{weaknnn}
    &~((\sigma_b(\overline{\vec{u}} ,\overline p ), \nabla \bxi))_{\Omega_b} - ((c_0 \overline p  + \alpha\nabla\cdot\overline{\vec{u}} , \partial_t q))_{\Omega_b} + ((k\nabla \overline p ,\nabla q))_{\Omega_b} \nn\\
    &+ 2\nu((\vec{D}(\overline{\vec{v}} ),\vec{D}(\bzeta)))_{\Omega_f} + (({\overline p ,(\bzeta-\bxi)\cdot\vec{e}_3}))_{\Gamma_I} - (({\overline{\vec{v}} \cdot\vec{e}_3,q}))_{\Gamma_I} \nn\\
    &- ((\overline{\vec{u}} \cdot\vec{e}_3,\partial_tq))_{\Gamma_I} + \beta (({\overline{\vec{v}} \cdot\btau,(\bzeta - \bxi)\cdot\btau}))_{\Gamma_I} + \beta(({\overline{\vec{u}} \cdot\btau, (\bzeta_t - \bxi_t)\cdot\btau}))_{\Gamma_I} \nn\\
    =&~ (d_0 ,q)_{\Omega_b}\big|_{t=0}  + ({\bu_0 \cdot\vec{e}_3,q})_{\Gamma_I}\big|_{t=0} \nn\\
    &~- \beta({\bu_0 \cdot\btau, (\bzeta - \bxi)\cdot\btau})_{\Gamma_I}\big|_{t=0}+ ((\vec{F}_b,\bxi))_{\Omega_b} + (\langle S,q\rangle)_{\Omega_b} + (\langle \vec{F}_f,\bzeta\rangle)_{\Omega_f}.
\end{align}
This shows the existence of a weak solution of the Biot-Stokes system with $\rho=(\rho_b,\rho_f)=\mathbf 0$, i.e., no inertia.

Next, observing our weak limit points (and subsequent identifications), and invoking weak-lower-semicontinuity of the norm, we obtain the estimate
\begin{align}\label{eidentn**}
  & \|\overline{\vec{u}}(t)\|_E^2 + c_0\|\overline p(t)\|_{0,b}^2 + \nn\\
    &+ \delta \int_0^t \left\|\overline{\bu}_t\right\|_E^2 d\tau   + k\int_0^t \|\nabla \overline p\|_{0,b}^2 d\tau + 2\nu\int_0^t\|\vec{D}(\overline{\vec{v}})\|_{0,f}^2d\tau + \int_0^t \beta\|(\overline{\vec{v}} - \overline{\vec{u}}_t)\cdot\btau\|_{\Gamma_I}^2d\tau \nn\\
    \lesssim &~ \|\bu_0\|_E^2 + c_0 \|p_0\|_{0,b}^2  \\ \nn &+\int_0^t[||\mathbf F_b||_{\vec{L}^2(\Omega_b)}^2+||\mathbf F_f||_{(\vec{H}^1_{\#,\ast}(\Omega_f)\cap \vec{V})'}^2+||S||^2_{[H^1_{\#,*}(\Omega_b)]'}]d\tau.
\end{align}

Finally, we observe the trace regularity for $\overline{\bu}_t$. Namely, our weak solution clearly has that $\gamma_0[(\overline{\mathbf v}-\overline{\mathbf u}_t) \cdot \btau] \in L^2(0,T; L^2(\Gamma_I))$ from the final estimate in \eqref{eidentn**}. Therefore $\gamma_0[\overline{\mathbf u}]_t \cdot \btau \in L^2(0,T;L^2(\Gamma_I))$ for $\btau=\mathbf e_1,\mathbf e_2$. 
\end{proof}

 \subsection{Vanishing Damping: $\delta \to 0$} \label{subsection:vanishing_damping}
 Now, we will consider the quasi-static Biot system when $\rho_b = \rho_f = 0$ and send $\delta \rightarrow 0$. The damped, quasi-static filtration system reduces to
\begin{align}\label{introsystem4}
- \mu\Delta\vec{u} - (\lambda + \mu)\nabla(\nabla\cdot\vec{u}) + \alpha \nabla p_b = \vec{F}_b, &\text{ in } \Omega_b \times (0,T),\\ 
    [c_0p_b + \alpha \nabla \cdot \vec{u}]_t - \nabla \cdot [k\nabla p_b] = S, &\text{ in } \Omega_b \times (0,T)\\
  -\nu\Delta \bv +\nabla p_f=\mathbf F_f;~~\nabla \cdot \bv = 0, & \text{ in } \Omega_f \times (0,T).
\end{align}
The relevant initial conditions for the limit system are reduced to
\begin{align} \label{equation:initial_condition2}
\bu(0)=\bu_0 \in \mathbf U,~~ [c_0p+\alpha\nabla \cdot \bu](0)=d_0 \in L^2(\Omega_b).
\end{align}

Once more, by writing $\delta \rightarrow 0$, we abuse the notation and mean to choose a sequence, if necessary a subsequence of, $(\delta_n)_{n \in \mathbb{N}}$ that converges to $0$. Our auxiliary result here is below. 
\begin{theorem}\label{th2}
     Suppose $d_0 \in L^2(\Omega_b)$, $\bu_0\in \mathbf U$, and $\mathbf F_f \in L^2(0,T;[\vec{H}^1_{\#,\ast}(\Omega_f)\cap \vec{V}]')$, $\mathbf F_b \in L^2(0,T;\mathbf L^2(\Omega_b))$, and $S \in L^2(0,T; [H^{1}_{\#,*}(\Omega_b)]')$. 
   Then for $\rho_b=\rho_f = \delta = 0$ and any $c_0 > 0$, the dynamics in \eqref{introsystem4} admit a weak solution $(\overline{\bu}, \overline{p}, \overline{\bv})^T \in \mathcal{V}_{sol}$. Furthermore, $\overline{\mathbf u} \in \mathcal V_b$ has  that $\gamma_0[\overline{\mathbf u}]_t \cdot \btau \in L^2(0,T;L^2(\Gamma_I))$ for $\btau=\mathbf e_1,\mathbf e_2$.
   The weak solution  satisfies the  energy inequality in \eqref{eidenttt} with $\delta=\rho_b = \rho_f = 0$. 
 \end{theorem}

\begin{proof}[Proof of Theorem \ref{th2}] Let $(\bu^ \delta, p^ \delta, \bv^ \delta)$ be a weak solution to the Biot system with $\rho_b = \rho_f = 0$ and $\delta > 0$ and the initial conditions 
\eqref{equation:initial_condition2}, i.e
\begin{align}\label{weakform2}
        &~ ((\sigma^E(\vec{u}^\delta), \nabla \bxi))_{\Omega_b}  +\delta ((\sigma^E(\bu_t^ \delta),\nabla \bxi))_{\Omega_b} - ((c_0 p^\delta + \alpha\nabla\cdot\vec{u}^ \delta, \partial_t q_b))_{\Omega_b} + ((k\nabla p^\delta,\nabla q_b))_{\Omega_b} \nn\\
        &+ 2\nu((\vec{D}(\vec{v}^\delta),\vec{D}(\bzeta)))_{\Omega_f} + (({p^\delta,(\bzeta-\bxi)\cdot\vec{e}_3}))_{\Gamma_I} - (({\vec{v}^\delta\cdot\vec{e}_3,q_b}))_{\Gamma_I} \nn\\
        &- (({\vec{u}^\delta\cdot\vec{e}_3,\partial_tq_b}))_{\Gamma_I} + \beta (({\vec{v}^\delta\cdot\btau,(\bzeta - \bxi)\cdot\btau}))_{\Gamma_I} + \beta(({\vec{u}^\delta\cdot\btau, (\bzeta_t - \bxi_t)\cdot\btau}))_{\Gamma_I} \nn\\
        =&~ (d_0,q_b)_{\Omega_b}\big|_{t=0} + ({\vec{u}_0\cdot\vec{e}_3,q_b})_{\Gamma_I}\big|_{t=0} \nn\\
        &~- \beta({\vec{u}_0\cdot\btau, (\bzeta - \bxi)\cdot\btau})_{\Gamma_I}\big|_{t=0} + ((\vec{F}_b,\bxi))_{\Omega_b} + (\langle S,q_b\rangle)_{\Omega_b} + (\langle \vec{F}_f,\bzeta\rangle)_{\Omega_f},
    \end{align}
It remains to pass $\delta \rightarrow 0$ and show that $\delta ((\sigma^E(\bu_t^ \delta),\nabla \bxi))_{\Omega_b} \longrightarrow 0$. 

Since the Lam\'e coefficients are constant, 
\begin{align} \label{eq:delta_to_0_weak_form}
    \delta((\sigma^E(\bu_t^\delta),\nabla \bxi))_{\Omega_b} = - \delta((\sigma^E(\bu^\delta), \nabla \bxi_t))_{\Omega_b} - \delta(\sigma^E(\bu^\delta), \nabla \bxi)_{\Omega_b}\big|_{t=0} + \delta(\sigma^E(\bu^\delta), \nabla \bxi)_{\Omega_b}\big|_{t=T}.
\end{align}
The third term on the RHS in \eqref{eq:delta_to_0_weak_form} vanishes due to $\bxi \in C_0^1([0,T); \textbf{U} \times H^1_{\#, *}(\Omega_b ))$ while the second term vanishes as $\bu^\delta(0) = \bu(0) = \bu_0$ for all $\delta$. The first term on the RHS of 
\eqref{eq:delta_to_0_weak_form} is bounded above 
\begin{align}
 ((\sigma^E(\bu^ \delta), \grad \bxi_t))_{\Omega_b} &\lesssim \int_{0}^T \left\|\nabla \bu^\delta(t)\right\|_{L^2(\Omega_b)}\left\|\nabla \bxi_t(t)\right\|_{L^2(\Omega_b)}dt.
\end{align}
Recall that $\bu^ \delta$ satisfies the energy estimate \eqref{eidentn**}  with upper bound independent of $\delta$ and therefore  $\{\bu^ \delta\}_{\delta > 0}$ is bounded in $L^ \infty(0,T; \vec{U})$, which implies that $\delta((\sigma^E(\bu^ \delta), \grad \bxi_t))_{\Omega_b} \longrightarrow 0$.

 From the energy estimate \eqref{eidentn**}, we have the following weak-$*$ (subsequential) limits and limit points, upon relabeling:
\begin{align*}
\bu^{\delta} \rightharpoonup^* &~\overline{\bu} \in L^{\infty}(0,T;\mathbf U)\\
p^{\delta}  \rightharpoonup^* &~\overline{p} \in L^{\infty}(0,T;L^2(\Omega_b)).
\end{align*}
From the dissipator, we  have the weak limits, which are identified as above (by uniqueness):
\begin{align*}
\bv^{\delta} \rightharpoonup &~\overline{\bv} \in L^{2}(0,T;\mathbf  H^1_{\#,*}(\Omega_f))\\
p^{\delta}  \rightharpoonup &~\overline{p} \in L^{2}(0,T;H^1_{\#,*}(\Omega_b)).
\end{align*}
 Thus, when $\delta \rightarrow 0$ , the convergence of other terms in \eqref{weakform2} follows as in the proof of Theorem \ref{theorem:vanished_inertia}. More specifically, equation \eqref{weakform2}  converges to weak form of the quasi-static Biot system

\begin{align}\label{weakform3}
        &~ ((\sigma^E(\overline{\vec{u}}), \nabla \bxi))_{\Omega_b} - ((c_0 \overline{p} + \alpha\nabla\cdot\overline{\vec{u}}, \partial_t q_b))_{\Omega_b} + ((k\nabla \overline{p},\nabla q_b))_{\Omega_b} \nn\\
        &+ 2\nu((\vec{D}(\overline{\vec{v}}),\vec{D}(\bzeta)))_{\Omega_f} + (({\overline{p},(\bzeta-\bxi)\cdot\vec{e}_3}))_{\Gamma_I} - (({\overline{\vec{v}}\cdot\vec{e}_3,q_b}))_{\Gamma_I} \nn\\
        &- (({\overline{\vec{u}}\cdot\vec{e}_3,\partial_tq_b}))_{\Gamma_I} + \beta (({\overline{\vec{v}}\cdot\btau,(\bzeta - \bxi)\cdot\btau}))_{\Gamma_I} + \beta(({\overline{\vec{u}}\cdot\btau, (\bzeta_t - \bxi_t)\cdot\btau}))_{\Gamma_I} \nn\\
        =&~ (d_0,q_b)_{\Omega_b}\big|_{t=0} + ({\vec{u}_0\cdot\vec{e}_3,q_b})_{\Gamma_I}\big|_{t=0} \nn\\
        &~- \beta({\vec{u}_0\cdot\btau, (\bzeta - \bxi)\cdot\btau})_{\Gamma_I}\big|_{t=0} + ((\vec{F}_b,\bxi))_{\Omega_b} + (\langle S,q_b\rangle)_{\Omega_b} + (\langle \vec{F}_f,\bzeta\rangle)_{\Omega_f}
    \end{align}
and $(\overline{\bu}, \overline{p}, \overline{\bv})$ is the solution to the quasi-static Biot system.

Lastly, we study the energy estimate \eqref{eidentn**}  when $\delta \rightarrow 0$. We will show that 
\begin{align}
    (\bv^{\delta}-\bu_t^{\delta})\cdot \btau \rightharpoonup &~ (\overline{\bv} - \overline{ \mathbf{u}}_t) \cdot \btau \in L^2(0,T;L^2(\Gamma_I)).
\end{align}
Recall that for each $\delta > 0$, $\gamma_0[\bu^\delta]_t \cdot \btau \in L^2(0,T;L^2(\Gamma_I))$ for $\btau=\mathbf e_1,\mathbf e_2$.
It follows from \eqref{eidentn**} that $(\vec{v}^\delta - \vec{u}^\delta_t)\cdot\btau$ is bounded in $L^2(0,T;L^2(\Gamma_I))$ and therefore converges weakly to some limit $\overline{h} \in L^2(0,T;L^2(\Gamma_I))$. Furthermore, $\bv^\delta\cdot\btau$ converges weakly to $\overline{\bv}\cdot\btau$. Therefore, $\bu_t^ \delta \cdot \btau \rightharpoonup \overline{\bv}\cdot\btau - \bar{h}$. On the other hand, choose $\bzeta = 0$ and $q_b = 0$; then we have
\begin{align}\label{weakform3*}
        \beta(({\overline{\vec{u}}\cdot\btau, - \bxi_t\cdot\btau}))_{\Gamma_I} \nn\\
        =& ~((\vec{F}_b,\bxi))_{\Omega_b} - ((\sigma^E(\overline{\vec{u}}), \nabla \bxi))_{\Omega_b} - (({\overline{p},(-\bxi)\cdot\vec{e}_3}))_{\Gamma_I}   -  \beta (({\overline{\vec{v}}\cdot\btau,- \bxi\cdot\btau}))_{\Gamma_I} 
    \end{align}
for all $\bxi \in C_0^\infty(0,T;H_\#^1(\Omega_b))$, and therefore, for all $\bxi \in L^2(0,T;H_\#^1(\Omega_b))$. Owing to the existence of a continuous right inverse for the tangential trace, the right hand-side of \eqref{weakform3*} is a continuous functional with respect to $\bxi\cdot \btau \in L^2(0,T;\gamma_0(H_\#^1(\Omega_b)))$. Thus, $\overline{\bu}_t\cdot\btau$ exists as an element of $L^2(0,T;\gamma_0(H_\#^1(\Omega_b))')$. Furthermore, for all $\bxi \in L^2(0,T;H_\#^1(\Omega_b))$, 
\begin{align}
\beta(({\vec{u}^{\delta}\cdot\btau, (\bzeta_t - \bxi_t)\cdot\btau}))_{\Gamma_I} \to &~ \beta(({\overline{\vec{u}}\cdot\btau, (\bzeta_t - \bxi_t)\cdot\btau}))_{\Gamma_I},
\end{align}
therefore $\gamma_0(\vec{u}_t^{\delta})\cdot\btau \rightharpoonup (\gamma_0\overline{\bu})_t \cdot \btau $ in $L^2(0,T;\gamma_0(H^{1}_\#(\Omega_b))') \supset L^2(0,T;L^2(\Gamma_I))$. By uniqueness of  limits, $\overline{\bv}\cdot\btau - \overline{h} = (\gamma_0\overline{\bu})_t \cdot \btau$. Combining this fact with the weak-lower-semicontinuity of norms, we obtain the following estimate

\begin{align}\label{eidentn6}
  & \|\overline{\vec{u}}(t)\|_E^2 + c_0\|\overline p(t)\|_{0,b}^2 + k\int_0^t \|\nabla \overline p\|_{0,b}^2 d\tau + 2\nu\int_0^t\|\vec{D}(\overline{\vec{v}})\|_{0,f}^2d\tau + \int_0^t \beta\|(\overline{\vec{v}} - \overline{\vec{u}}_t)\cdot\btau\|_{\Gamma_I}^2d\tau \nn\\
    \lesssim &~ \|\bu_0\|_E^2 + c_0 \|p_0\|_{0,b}^2 +\int_0^t[||\mathbf F_b||_{\vec{L}^2(\Omega_b)}^2+||\mathbf F_f||_{(\vec{H}^1_{\#,\ast}(\Omega_f)\cap \vec{V})'}^2+||S||^2_{[H^1_{\#,*}(\Omega_b)]'}]d\tau.
\end{align}
This concludes the proof of Theorem \ref{th2}. 
\end{proof}

\subsection{Brief Remarks on $c_0 \searrow 0$}

Finally, we mention the case of allowing $c_0 =0$. This case may be treated in the same manner---and more straightforwardly---as the other limits. Indeed, the quasi-static dynamics allow for limit passage in the weak formulation as $c_0 \searrow 0$, since the pressure retains $p \in L^2(0,T;H^1_{\#,*}(\Omega_b))$ in the limit. The construction follows identically to that given in \cite{AGW} in the inertial case and hence we do not include it here. It is central to note that when $c_0=0$, we have the implicit compatibility condition that $\alpha \nabla \cdot \bu_0=d_0 \in L^2(\Omega_b)$. We thus conclude the validity of our main theorem, Theorem \ref{th:main1}. 

\begin{remark} We mention that our main result (and its proof) does not address additional time-regularity of the function $\bu \in L^{\infty}(0,T;\mathbf U)$. On the other hand, we require initializing the dynamics with $\bu_0 \in \mathbf U$. Our weak formulation in \eqref{weakform} carries the initial conditions on the right hand side. We do expect additional time continuity of both the quantity $\bu$ as well as the fluid content $[c_0p+\alpha \nabla \cdot \bu]$ in certain spaces, however, this is a delicate matter in these coupled models. Indeed, for the Biot dynamics alone, this was a central issue in \cite{bw}. Recently, the time regularity of solutions in the inertial cases was investigated for the purposes of uniqueness arguments in \cite{AW}; future work will address these issues for this fully quasi-static filtration case. 
\end{remark}

\section*{Appendix}
In this section, we elaborate on the semigroup approach to prove well posedness  of Biot-Stoke system for the case $\rho_b$, $\rho_f$, $\delta$ and $c_0$ are positive as mentioned in section \ref{subsection:Ancillary_Result}. To define the semigroup generator for $\mathcal A_{\delta}$ we  first address the elimination of the pressure, following the strategy in \cite{AGW}.

Suppose  $\mathbf F_f \in \mathbf V$ pointwise in time. Then we consider the fluid-pressure sub-problem, again pointwise in time:
\begin{equation}\label{bigellip}
    \begin{cases}
        \Delta p_f = 0 &\text{ in } \Omega_f,\\
        \partial_{\vec{e}_3}p_f = \nu \Delta\vec{v} \cdot\vec{e}_3 &\text{ on } \Gamma_f,\\
        p_f = p_b + 2\nu\vec{e}_3\cdot \vec{D}(\vec{v})\vec{e}_3  &\text{ on } \Gamma_I,\\
        p_f ~\text{ is }~x_1\text{-periodic} ~\text{ and }~ x_2\text{-periodic}& \text{ on } \Gamma_{\text{lat}}.
    \end{cases}
\end{equation}
Laplace's equation is obtained from  taking the divergence of the  fluid equation in \eqref{stokes1}; the condition on $\Gamma_f$ is obtained as the inner product of the  fluid equation with $\vec{e}_3$, restricted to $\Gamma_f$; and the condition on $\Gamma_I$ is simply \eqref{IC4}. Here, we recall that $\vec{e}_3 = \vec{n}_f$.

With the representation \eqref{bigellip} in mind, we define the Neumann and Dirichlet Green's maps $N_f$ and $D_I$ by
\[
    \phi \equiv N_fg \iff 
    \begin{cases}
        \Delta \phi = 0 & \text{in} ~\Omega_f;\\
        \partial_{\vec{e}_3} \phi = g &  \text{on} ~ \Gamma_f;\\
        \phi = 0 &  \text{on} ~ \Gamma_I;
    \end{cases}
    \quad \quad
    \psi \equiv D_I h \iff
    \begin{cases}
        \Delta \psi = 0 &  \text{in} ~\Omega_f;\\
        \partial_{\vec{e}_3} \psi = 0 &  \text{on} ~ \Gamma_f;\\
        \psi = h &  \text{on} ~\Gamma_I;
    \end{cases}
\]
where we have suppressed the  lateral periodicity of the pressure. 
The maps $N$ and $D$ are well-defined via classical elliptic theory for sufficiently regular data; in this setting, as periodicity effectively eliminates corners and preserves the  boundary regions as ``separated" and smooth. The mappings can then  be extended through transposition for weaker data \cite{redbook}:
\begin{align*}
    N_f &\in \mathcal{L}(H^s(\Gamma_f),  H^{s + \frac{3}{2}}(\Omega_f)),\quad D_I \in \mathcal{L}(H^s(\Gamma_I), H^{s + \frac{1}{2}}(\Omega_f)),\quad s \in \mathbb{R}.
\end{align*}

Then for given $p_b$ and $\vec{v}$ in appropriate spaces, we can write $$p_f = \Pi_1p_b + \Pi_2\vec{v} + \Pi_3\vec{v},~ \text{ where},$$
\begin{equation}\label{PiOps}
    \Pi_1p_b = D_I[(p_b)_{\Gamma_I}],\quad  \Pi_2\vec{v} = D_I\big[(\vec{e}_3 \cdot [2\nu\vec{D}(\vec{v})\vec{e}_3])_{\Gamma_I}\big], \quad \Pi_3\vec{v} = N_f[(\nu\Delta\vec{v} \cdot \vec{e}_3)_{\Gamma_f}].
\end{equation}
Finally, for $i=1,2,3$, we denote $G_i = -\nabla\Pi_i$, as invoked in \eqref{A}---for convenience, we will work directly with the $\Pi_i$ below. 
We then consider a Cauchy problem which captures the dynamics of the full Biot-Stokes system. 

Namely: find $\vec{y} = [\vec{u},\vec{w},p_b,\vec{v}]^T \in \mathcal{D}(\cA_{\delta})$ (resp. $X$) such that
\begin{equation}\label{cauchy}
   \dot{ \vec{y}}(t) = \cA_{\delta} \vec{y}(t)+\mathcal F(t); \quad \vec{y}(0) = [\vec{u}_0,\vec{u}_1,p_0,\vec{v}_0]^T \in \mathcal{D}(\cA_{\delta}) \text{ (resp. } X),
\end{equation}
where $\mathcal F(t) = [\mathbf 0, \mathbf F_b(t), S(t), \mathbf F_f(t)]^T$ (of appropriate regularity)
and the action of the  operator $\cA_{\delta}:\mathcal{D}(\cA_{\delta})\subset X \to X$ is provided by
\begin{equation}\label{A}
    \cA_{\delta} \equiv \begin{pmatrix}
        0 & \vec{I} & 0 & 0\\
        -\cE_0 & -\delta \cE_{0} & -\alpha\rho_b^{-1} \nabla & 0\\
        0 & -\alpha c_0^{-1} \nabla\cdot & -A_0 & 0\\
        0 & 0 & \rho_f^{-1}G_1 & \rho_f^{-1}[\nu\Delta + G_2 + G_3]
    \end{pmatrix}.
\end{equation}
The differential action of the generator $\mathcal A_\delta$ can be rewritten as: \begin{equation}\label{diffaction}
   \mathcal A_{\delta}\mathbf y\equiv  \begin{pmatrix}
        \vec{w}\\
        -\mathcal{E}_0\vec{u} -\delta \cE_{0}\vec{w}-\alpha \rho_b^{-1}\nabla p_b\\
        -\alpha c_0^{-1}\nabla\cdot \vec{w} -A_0p_b\\
        \rho_f^{-1}\nu \Delta \vec{v} - \rho_f^{-1}\nabla\pi(p_b,\bv)
    \end{pmatrix}
    =
    \begin{pmatrix}\bw_1^*\\\bw_2^*\\q^*\\\mathbf f^*\end{pmatrix} \in X, \quad \vec{y} = [\vec{u},\vec{w},p_b,\vec{v}]^T \in \mathcal{D}(\mathcal{A}_\delta),
\end{equation}
where we denote $\pi = \pi(p_b,\vec{v})\equiv \Pi_1p_b+\Pi_2 \mathbf v+\Pi_3\mathbf v \in L^2(\Omega_f)$ as a solution to the following elliptic problem, for given $p_b$ and $\vec{v}$ emanating from $\vec{y} \in \mathcal{D}(\mathcal{A}_{\delta})$ and thus having adequate regularity to form the RHS below:
\begin{equation}\label{pi}
    \begin{cases}
        \Delta \pi = 0 &\in L^2(\Omega_f),\\
        \partial_{\vec{e}_3}\pi = \nu \Delta \vec{v} \cdot\vec{e}_3& \in H^{-3/2}(\Gamma_f),\\
        \pi = p_b + 2\nu\vec{e}_3\cdot \vec{D}(\vec{v})\vec{e}_3  &\in H^{-1/2}(\Gamma_I),\\
        \pi  \in H^{-1/2}_{\#}(\partial \Omega_f).
    \end{cases}
\end{equation}
The above action then generates the  {\it elliptic}  system:
$$\begin{cases}
        -\mathcal{E}_{0}\bu-\delta \mathcal{E}_{0}\bw - \alpha\rho_{b}^{-1}\nabla p_b = \bw_2^\ast & \in \vec{L}^2(\Omega_b),\\
        -A_0p_b = q^\ast + \alpha c_{0}^{-1}\nabla\cdot{\bw_1^*} &\in L^2(\Omega_b),\\
        \rho_f^{-1}[\nu \Delta \vec{v} - \nabla\pi] = \mathbf f^\ast &\in \vec{V},\\
        \nabla\cdot\vec{v} = 0,\\
     p_b\big|_{\Gamma_b} = 0,\quad \vec{v}|_{\Gamma_f} = 0, \quad \vec{u}|_{\Gamma_b} = 0\\
        \vec{v}\cdot\vec{e}_3 +  k\partial_{\vec{e}_3} p_b = \vec{w}\cdot\vec{e}_3 &\in H^{-1/2}(\Gamma_I),\\
        \beta\vec{v}\cdot\btau +  \btau \cdot [2\nu\vec{D}(\vec{v})-\pi]\vec{e}_3 =  \beta \vec{w}\cdot\btau&\in H^{-1/2}(\Gamma_I),\\
        \sigma_b\vec{e}_3 = [2\nu\mathbf D(\bv)-\pi]\mathbf e_3 & \in \vec{H}^{-1/2}(\Gamma_I),\\
        p_b = - 2\nu\vec{e}_3\cdot\mathbf D(\bv)\mathbf e_3+\pi &\in H^{-1/2}(\Gamma_I).
    \end{cases} $$
Again, $\pi=\pi(p_b,\bv)$, as given in \eqref{pi}, and we recall that 
$$ \sigma_b = \sigma_b(\vec{u},\delta \bu_t, p_b) \equiv  \sigma^E(\vec{u}) +\delta \sigma^E(\bu_t)- \alpha p_b\vec{I}.$$

Thence we can define the operator domain of $\mathcal A_{\delta}$; all restrictions below are interpreted as traces, in the generalized sense.
\begin{definition}[Domain of $\cA_{\delta}$] \label{diffdomain} Let $\vec{y}\in X$. Then $\vec{y} = [\vec{u},\vec{w},p,\vec{v}]^T \in \mathcal{D}(\mathcal{A}_{\delta})$ if and only if the following bullets hold:
\begin{itemize}\setlength\itemsep{.03cm}
    \item $\vec{u},\vec w \in \vec{U}$  with  $\mathcal{E}_{0}(\vec{u}),\mathcal{E}_{0}(\vec{w}) \in \vec{L}^2(\Omega_b)$ (so that $[{\sigma_b(\vec{u}+\delta \vec{w})\mathbf n_f}]\big|_{\partial\Omega_b} \in \vec{H}^{-1/2}(\partial \Omega_b)$);
    
    \item $p \in  H_\#^1(\Omega_b)$ with $A_0p \in L^2(\Omega_b)$ (so that $\left.\partial_{\mathbf n_f}p\right|_{\partial\Omega_b} \in {H}^{-1/2}(\partial \Omega_b)$);
    
        \item $[{\sigma_b\mathbf n_f}] \in \vec{H}_{\#}^{-1/2}(\partial \Omega_b)$ (then  
       $\sigma^E(\bu+\delta \bw)\mathbf n_f\cdot \mathbf n_f \in H^{-1/2}_{\#}(\partial \Omega_b)$);
       \item  $\partial_{\mathbf n_f} p \in H^{-1/2}_{\#}(\partial \Omega_b)$
    
    \item $\vec{v} \in \vec{H}_\#^1(\Omega_f) \cap \vec{V}$ with $\vec{v}|_{\Gamma_f} = \vec{0}$;
    \item There exists $\pi \in L^2(\Omega_f)$ such that ~$\nu\Delta \vec{v} - \nabla\pi \in \vec{V},$ where $\pi$ is as in \eqref{pi} (and so
 $\left.\sigma_f\mathbf n_f\right|_{\partial \Omega_f}\in \vec{H}^{-\frac{1}{2}}(\partial\Omega_f)$ and  $\left.\pi\right|_{\partial \Omega_f} \in H^{-\frac{1}{2}}(\partial\Omega_f)$;
   
    \item $2\nu \mathbf D(\bv) \big|_{\partial \Omega_f} \in \mathbf H^{-1/2}_{\#}(\partial \Omega_f)$ and ~$\pi \big|_{\partial\Omega_f} \in H^{-1/2}_{\#}(\partial \Omega_f)$;
    
    \item  $\left.\Delta \vec{v}\cdot\mathbf n_f\right|_{\partial \Omega_f} \in H^{-3/2}(\partial\Omega_f)$);
    
    \item  $\restri{(\vec{v} - \vec{w})\cdot\vec{e}_3} = \restri{-k\nabla p\cdot\vec{e}_3} \in H^{-1/2}(\Gamma_I)$;
    
    \item $\restri{\beta(\vec{v}-\vec{w})\cdot\btau} = \restri{-\btau\cdot\sigma_f\vec{e}_3} \in H^{-1/2}(\Gamma_I)$;
    
    \item $\restri{-\vec{e}_3 \cdot \sigma_f\vec{e}_3} = \restri{p} \in H^{-1/2}(\Gamma_I)$;
    
    \item $\restri{\sigma_b\vec{e}_3} = \restri{\sigma_f\vec{e}_3} \in \vec{H}^{-1/2}(\Gamma_I)$.

\end{itemize}
\end{definition}
\noindent The above terms with  traces in  Sobolev spaces of negative indices are well-defined (as extensions) when $\vec{y} \in \mathcal{D}(\cA_{\delta})$.

We remark that this generator is the natural companion to that presented in \cite{AGW,AW} for the undamped ($\delta=0$) dynamics. Obtaining generation for this operator $\mathcal A_{\delta}$ defined above follows mutatis mutandis for the action and underlying spaces as the Biot-Stokes dynamics with $\delta=0$. This is tantamount to simply adding strong (visco-elastic) damping to an existing hyperbolic generator.

\end{document}